\documentclass[10pt]{amsart}

\usepackage[all]{xy}
\usepackage{amsmath, amssymb, hyperref}
\usepackage{tikz}

\newtheorem{thm}{Theorem}[section]
\newtheorem{prop}[thm]{Proposition}
\newtheorem{cor}[thm]{Corollary}
\newtheorem{lem}[thm]{Lemma}

\theoremstyle{remark}
\newtheorem{remark}[thm]{Remark}

\theoremstyle{definition}
\newtheorem{definition}[thm]{Definition}

\newcommand*\isom{\xrightarrow{\sim}}

\makeatletter
\renewcommand*\env@matrix[1][*\c@MaxMatrixCols c]{%
  \hskip -\arraycolsep
  \let\@ifnextchar\new@ifnextchar
  \array{#1}}
\newcommand{\pair}[1]{\langle#1\rangle}
\newcommand{\divisor}{\operatorname{div}}
\newcommand{\bdivisor}{\operatorname{b-div}}
\newcommand{\ord}{\operatorname{ord}}
\newcommand{\Pic}{\operatorname{Pic}}
\newcommand{\Div}{\operatorname{Div}}

\newcommand{\cDiv}{\operatorname{c-Div}}
\newcommand{\cPic}{\operatorname{c-Pic}}
\newcommand{\bDiv}{\operatorname{b-Div}}
\newcommand{\bPic}{\operatorname{b-Pic}}

\newcommand{\SL}{\operatorname{SL}}

\newcommand{\Princ}{\operatorname{Prin}}
\newcommand{\Map}{\operatorname{Map}}
\newcommand{\Dir}{\operatorname{Vec}}

\newcommand{\Zh}{\operatorname{Zh}}

\def\qq{\mathbb{Q}}

\def\rr{\mathbb{R}}
\def\zz{\mathbb{Z}}
\def\cc{\mathbb{C}}

\def\mm{\mathcal{M}}
\def\HH{\mathbb{H}}
\def\ll{\mathcal{L}}
\def\bb{\mathcal{B}}

\def\I{\mathcal{I}}
\def\oo{\mathcal{O}}

\def\ee{\mathcal{E}}
\def\xx{\mathcal{X}}

\def\deldelbar{\partial \bar{\partial}}
\def\Ar{\mathrm{Ar}}
\def\omar{\omega_{\Ar}}
\def\opar{\oo(P)_{\Ar}}
\def\oqar{\oo(Q)_{\Ar}}

\def\d{\mathrm{d}}
\def\Lbar{\bar{L}}
\def\Ebar{\bar{E}}
\def\llbar{\bar{\ll}}

\def\mmbar{\bar{\mm}}

\def\sing{s}

\def\a{a}

\def\can{\mathrm{can}}

\numberwithin{equation}{section}

\newcounter{nootje}
\begin{document}

\title[Chern-Weil theory for the family Arakelov metric]{Chern-Weil theory for line bundles with the family Arakelov metric}

\author{Michiel Jespers and Robin de Jong}

\begin{abstract} We prove a result of Chern-Weil type for canonically metrized line bundles on one-parameter families of smooth complex curves. Our result generalizes a result due to J.I.~Burgos Gil, J.~Kramer and U.~K\"uhn that deals with a  line bundle of Jacobi forms on the universal elliptic curve over the modular curve with full level structure,  equipped with the Petersson metric. Our main tool, as in the work by Burgos Gil, Kramer and K\"uhn, is the notion of a b-divisor. 
\end{abstract}

\maketitle


\thispagestyle{empty}

\section{Introduction}

To motivate the results in this paper, we take as a starting point D.~Mumford's famous result \cite{hi} from the 1970s, that vector bundles on algebraic varieties equipped with so-called \emph{good} hermitian metrics satisfy Chern-Weil theory, and that automorphic vector bundles on pure, non-compact Shimura varieties, equipped with their canonical invariant metric, are good. 

More precisely, let $S$ denote a smooth connected complex algebraic variety, and let $S \hookrightarrow \xx$ denote a compactification of $S$ where $\xx$ is smooth, and where the boundary $\xx \setminus S$ is a normal crossings divisor on $\xx$. Let $\bar{E}=(E,\|\cdot\|)$ denote a smooth hermitian vector bundle on $S$, and assume that the metric $\|\cdot\|$ on $E$ is good in the sense of Mumford on $\xx$. The metric  on $E$ then determines a canonical extension $[\bar{E},\xx] $ of $E$ as a vector bundle over $\xx$, by asking that local sections have at most logarithmic growth across the boundary of $S$ in $\xx$. The formation of Mumford's canonical extension is easily seen to be compatible with birational morphisms $\xx' \to \xx$ of compactifications of $S$ that restrict to the identity on $S$.

The first main result of \cite{hi} is that for all non-negative integers $k$ the Chern form $c_k(\bar{E})$, viewed as a current (singular $(k,k)$-form)  on the compactification $\xx$, is integrable. More precisely, the cohomology class $[c_k(\bar{E})]$ of the current $c_k(\bar{E})$ on $\xx$ coincides with the Chern class $c_k([\Ebar,\xx])$ of the canonical extension of $\Ebar$ over $\xx$. As a special case, then, we note that when $\Lbar_1,\ldots,\Lbar_n$ are smooth hermitian  line bundles on $S$ which are good on $\xx$, with $n=\dim(S)$, one has the identity
\begin{equation} \label{cwtheory} \left[\Lbar_1,\xx\right] \cdots  \left[\Lbar_n,\xx\right]
= \int_S c_1(\Lbar_1) \wedge \ldots \wedge c_1(\Lbar_n) 
\end{equation}
in $\zz$. We refer to (\ref{cwtheory}) as a result of Chern-Weil type for  $\Lbar_1,\ldots,\Lbar_n$ on $S$. 

The second main result of \cite{hi} is that if $S$ is taken to be a pure, non-compact Shimura variety, with $S \hookrightarrow \xx$ a toroidal compactification of $\xx$, and with $\bar{E}$ an automorphic vector bundle on $S$ equipped with its canonical invariant metric, the metric on $\bar{E}$ is good on $\xx$. In particular, combining with the first main result  we see that the smooth hermitian vector bundle $\bar{E}$ satisfies Chern-Weil theory on $\xx$. 

It seems natural  to try to extend the equality in (\ref{cwtheory}) to a broader range of geometrically relevant situations, including cases where the metric on the open variety $S$ is not necessarily good on some compactification $\xx$. For example, in the recent paper \cite{bkk} by J.I.~Burgos Gil, J.~Kramer and U.~K\"uhn the following situation is considered. Let $N \geq 3$ be an integer, and take $S$ to be the complex algebraic surface $E=E(N)$, the universal elliptic curve over the open modular curve $Y(N)$ with full level-$N$ structure. Let $O$ denote the image in $E$ of the zero-section of $E$ over $Y(N)$. Let $\xx$ be the minimal regular model $\ee=\ee(N)$ of $E$ over the smooth completion $X(N)$ of $Y(N)$, and take $\Lbar$ to be the line bundle $L=\oo_E(8O)$ on $E$ equipped with the smooth hermitian metric that gives the global section $\vartheta_{1,1}^8$ of $L$ its standard Petersson norm $\|\vartheta_{1,1}\|^8$. Here $\vartheta_{1,1}$ denotes Riemann's elliptic theta function with characteristics $[1/2,1/2]$. 

In \cite{bkk} an analysis is carried out of the degeneration behavior of the metric on $\Lbar$  near the boundary divisor $B$ of $E$ in $\ee$. The divisor $B$ is a disjoint union of $N$-gons, and in particular is a normal crossings divisor on $\ee$. The analysis in \cite{bkk} shows that the smooth hermitian line bundle $\Lbar^{\otimes N}$ has logarithmic growth on $\ee$ only away from the set $\Sigma \subset \ee$ of singular points of $B$. 

By Hartogs's theorem, as $\Sigma$ has codimension two in $\ee$, the canonical Mumford extension of $\Lbar^{\otimes N}$ over $\ee \setminus \Sigma$ has a unique extension as a line bundle over $\ee$, so that we continue to have a canonical extension $[\Lbar^{\otimes N},\ee]$ of $L^{\otimes N}$ over $\ee$ dictated by the metric on $\Lbar$. This extension is baptized the \emph{Mumford-Lear extension} of $\Lbar^{\otimes N}$ over $\ee$ in \cite{bkk}, thus referring to both \cite{hi} and D.~Lear's PhD thesis \cite{lear}, where canonical extensions of canonically metrized line bundles over compactifications are studied in the context of asymptotic Hodge theory. 

A naive thought might now be, since $\Sigma$ is only of codimension two, and since we know Mumford's result for pure Shimura varieties, that Chern-Weil theory  (\ref{cwtheory}) might continue to hold in this case. Indeed, note that the surface $E$ is an example of a mixed Shimura variety, that the surface $\ee$ is a toroidal compactification of it, and that $\Lbar$ is an automorphic line bundle with canonical invariant metric on $E$. 

The crucial observation of \cite{bkk} is however that, perhaps surprisingly at first sight, the equality in (\ref{cwtheory}) fails to hold in this setting. To be precise, let $p_N$ denote the number of cusps of $X(N)$. Then in \cite[Proposition~4.10]{bkk} it is computed that 
\begin{equation} \label{selfint} \left[ \Lbar^{\otimes N}, \ee \right]\cdot
\left[ \Lbar^{\otimes N}, \ee \right] = \frac{16}{3} N(N^2+1)p_N  \, ,
\end{equation}
whereas the proof of \cite[Theorem~5.2]{bkk} shows that
\begin{equation} \label{integral}
\int_E c_1(\Lbar) \wedge c_1(\Lbar) = \frac{16}{3} Np_N  \, . 
\end{equation}
As for an explanation of this discrepancy, a closer study of the singularities of the metric on $\Lbar$ near the points in $\Sigma$ shows that the formation of the Mumford-Lear extension $[\Lbar^{\otimes N},\ee]$ is not functorial with respect to simple blow-ups of the surface $\ee$ if  the center of the blow-up is taken in $\Sigma$. It is observed in \cite{bkk} that this non-functoriality is closely related to the phenomenon of ``height jumping'' that hermitian metrics may display near subsets of codimension two, as first noted in examples related to moduli spaces of curves by R.~Hain \cite{hain_normal}. Once one sees that geometrically relevant metrics may display a non-trivial height jump, it is already less of a surprise that Chern-Weil theory can fail, also in cases of geometric interest. 

The next question is then, whether Chern-Weil theory (\ref{cwtheory}) can be restored in any natural sense. There seems to be nothing we can change about the integral over the open part occurring on the right hand side of (\ref{cwtheory}). However, the previous remarks indicate that it may not be natural, on the left hand side of (\ref{cwtheory}) where one computes an intersection product, to work with a single chosen model of $E$. It seems that one can only fully capture the information stored in the singularities of the metric near the singular points of the boundary divisor by actually bringing the category $\I$ of \emph{all} regular models $\xx$ of $E$  into the picture. On each single model $\xx$  one might find a Mumford-Lear extension $[\Lbar^{\otimes N},\xx]$ as above, however, the self-intersection $\left[ \Lbar^{\otimes N}, \xx \right]\cdot \left[ \Lbar^{\otimes N}, \xx \right]$ may change as one varies $\xx$ over the category $\I$ (and in fact, it does change). A new idea might then be to think that perhaps the integral in (\ref{integral}) may be computed as the limit, in a suitable sense, of all self-intersections $\left[ \Lbar^{\otimes N}, \xx \right]\cdot \left[ \Lbar^{\otimes N}, \xx \right]$ as one varies $\xx$ over $\I$.

The beautiful main point of \cite{bkk} is that this new idea indeed turns out to work. The technical issue is then, to develop the necessary birational intersection theory to make sense of the limit. A central concept in \cite{bkk} is that of a \emph{b-divisor}, as originally introduced by V.~Shokurov \cite{sh-prel} \cite{sh-3fold}.  For our purposes, the following variant of this concept will be useful. Let again $S$ be a connected smooth complex algebraic variety. A regular \emph{nc}-model of $S$ is any smooth completion $S \hookrightarrow \xx$ of $S$ such that the boundary $\xx \setminus S$ is a normal crossings divisor on $\xx$. 

Let $\I$ denote any cofiltered category of regular \emph{nc}-models of $S$. Here it is understood that a morphism in $\I$ is a birational morphism that restricts to the identity on $S$. Each hom-set in $\I$ is thus either empty or consists of one element. For each $\xx \in \I$ we let $\Div(\xx)$ denote the group of (Weil) divisors on $\xx$, and we write $\Div_\qq(\xx) = \Div(\xx) \otimes_\zz \qq$.  The groups $\Div_\qq(\xx)$ form an inverse system with respect to proper pushforward of divisors. In \cite{bkk}, and also for us in this paper, a \emph{b-divisor on $\I$} is by definition an element of the projective limit
\[ \bDiv_\qq(\I) = \varprojlim_{\xx \in \I} \Div_\qq(\xx) \]
in the category of abelian groups.  Thus, an element $\mathbb{D}$ of $\bDiv_\qq(\I)$ can be viewed as a tuple $(D_\xx)_{\xx \in \I}$ of elements of $\Div_\qq(\xx)$ for $\xx \in \I$, satisfying the compatibility condition that $\varphi_* D_{\xx'} = D_{\xx}$ whenever $\varphi \colon \xx' \to \xx$ is a morphism in the category~$\I$.

We have a natural (partial) multilinear intersection form on the abelian group $\bDiv_\qq(\I)$. Namely, let $n=\dim(S)$ and let $\mathbb{D}_1,\ldots,\mathbb{D}_n$ be b-divisors on $\I$.  Write $\mathbb{D}_1=(D_{1,\xx})_{\xx \in \I}, \ldots, \mathbb{D}_n=(D_{n,\xx})_{\xx \in \I}$, and note that on each model $\xx$ we have an intersection product $D_{1,\xx}  \cdots D_{n,\xx}$ in $\qq$. We then put
\begin{equation} \label{defproduct}  \mathbb{D}_1 \cdots \mathbb{D}_n = \lim D_{1,\xx}  \cdots D_{n,\xx} \, , 
\end{equation}
where the limit is taken in the sense of nets. As we will see, finiteness or even existence of the limit is not to be expected for general tuples of b-divisors; rather strong conditions on the growth of the multiplicities in $D_{i,\xx}$ of the boundary components of $S$ in $\xx$ are needed, for example. 

We say that a tuple $(\mathbb{D}_1,\ldots,\mathbb{D}_n)$ is \emph{integrable} if the limit (\ref{defproduct}) exists in $\rr$. A b-divisor $\mathbb{D}$ is called integrable if the tuple $(\mathbb{D},\ldots,\mathbb{D})$ is integrable. A main guiding principle is that tuples of ``geometrically relevant'' b-divisors should be integrable. For example, in the main text below we will see natural sufficient conditions for integrability in the case $n=2$.

Now let $\llbar=(\ll,\|\cdot\|)$ be a smooth hermitian line bundle on the smooth connected complex algebraic variety $S$. Our aim is to associate to $\llbar$ a b-divisor on $\I$. Assume that on each model $\xx$ in the category $\I$ some sufficiently large tensor power of $\llbar$ has logarithmic growth away from codimension two (we refer to the main text below for precise definitions). We will briefly say in this case that $\llbar$ ``admits all Mumford-Lear extensions'' in $\I$. For each $\xx \in \I$, the canonical extension $\left[ \llbar,\xx \right]$ determined by the metric is then a $\qq$-line bundle on $\xx$. Let $s$ be a non-zero rational section of $\ll$, and for each $\xx \in \I$ let $\divisor_\xx(s)$ denote its $\qq$-divisor on $\xx$ as a rational section of the Mumford-Lear extension $\left[ \llbar,\xx \right]$. It is shown in \cite[Proposition~3.15]{bkk} that the resulting tuple $\bdivisor(s)=(\divisor_\xx (s))_{\xx \in \I}$ is a b-divisor on $\I$, i.e.\ an element of $\bDiv_\qq(\I)$. 

We can now formulate the main result of \cite{bkk}. 
\begin{thm} (Burgos Gil, Kramer and K\"uhn,  \cite[Theorem~5.2]{bkk}) \label{bkkthm} Let $N \geq 3$ be an integer. Let $E=E(N)$ denote the universal elliptic curve over the modular curve $Y(N)$ and let $\Lbar$ denote the line bundle $L=\oo_E(8O)$ on $E$ equipped with the metric that gives the global section $\vartheta_{1,1}^8$ of $L$ its Petersson norm. Then $\Lbar$ admits all Mumford-Lear extensions in the category $\I$ of all regular \emph{nc}-models of $E$ over $X(N)$. The b-divisor
\[ \bdivisor(\vartheta_{1,1}^8) \in \bDiv_\qq(\I) \]
is integrable, and the equality
\begin{equation} \label{cwtheoryell} \bdivisor(\vartheta_{1,1}^8) \cdot \bdivisor(\vartheta_{1,1}^8) = \int_E c_1(\Lbar) \wedge c_1(\Lbar)  
\end{equation}
holds in $\qq$.
\end{thm}
Our aim in this paper is to generalize Theorem \ref{bkkthm} in the direction of families of complex smooth projective curves of arbitrary genus $g \geq 1$ fibered over a smooth complex curve. The line bundles we will consider are equipped with the fiberwise canonical (Arakelov) metric, as introduced by S. Y. Arakelov \cite{ar}, and we will assume that the minimal regular model of the family is semistable. As is the case with the family of examples considered in \cite{bkk}, we find that the singularities of the Arakelov metric develop a height jump around the singular points of the boundary divisor, and that Chern-Weil theory fails in general on the minimal regular model (and in fact on any regular model).  However, Chern-Weil theory will be satisfied  in the b-divisorial sense. 

We describe our results in more detail. Let $C^0$ be a smooth connected complex algebraic curve and let $\pi \colon S \to C^0$ be a smooth proper morphism with connected fibers of dimension one and of genus $g \geq 1$. Let $C$ be the smooth completion of $C^0$.
We consider the category $\I$ of all regular \emph{nc}-models $S \hookrightarrow \xx$ of $S$ such that the morphism $\pi \colon S \to C^0$ extends to a (then unique) morphism $\xx \to C$. The objects of $\I$ are called regular \emph{nc}-models of $S$ over $C$. The category $\I$ is cofiltered. As above, we have an associated group of b-divisors $\bDiv_\qq(\I)$ on $\I$, equipped with a partial intersection pairing given by (\ref{defproduct}) with $n=2$.

It is straightforward to check that we have a good notion of principal b-divisors inside $\bDiv_\qq(\I)$. We then write $\bPic_\qq(\I)$ for the quotient of $\bDiv_\qq(\I)$ by the subgroup of principal b-divisors. The partial intersection pairing on $\bDiv_\qq(\I)$ descends to a partial intersection pairing on $\bPic_\qq(\I)$, and in particular we have a notion of integrability for (pairs of) classes in $\bPic_\qq(\I)$. 

Let $\llbar=(\ll,\|\cdot\|)$ be a smooth hermitian line bundle on $S$ that admits all Mumford-Lear extensions in the category $\I$. Above we have seen a construction that to each non-zero rational section $s$ of $\ll$ associates a b-divisor $\bdivisor(s)$ in $\bDiv_\qq(\I)$. Not surprisingly, the class of $\bdivisor(s)$ in $\bPic_\qq(\I)$ is independent of the choice of $s$, and hence one has canonically associated to $\llbar$ a class in $\bPic_\qq(\I)$. We call this class the \emph{b-Mumford-Lear extension} of $\llbar$, and denote it by $\left[ \llbar \right]$.

In most of our discussion, the smooth hermitian line bundle $\llbar$ will be picked from a collection of canonical ones that one has on $S$ by Arakelov's work \cite{ar} \cite {fa}. When $M$ is a complete connected complex curve of positive genus, and $P \in M$ is a point, then the line bundle $\oo(P)$ has a canonical metric determined by Arakelov's Green's function on $M$. By varying $P$ through $M$ this assignment induces a smooth hermitian metric $\|\cdot\|$ on the line bundle $\oo(\Delta)$ on the self-product $M \times M$, where $\Delta \subset M \times M$ is the diagonal divisor. By Poincar\'e's residue isomorphism the metric $\|\cdot\|$  induces a canonical smooth hermitian metric on the cotangent line bundle $\omega \isom \oo(-\Delta)|_{\Delta}$ of $M$. We refer to Section \ref{sec:metric} below for precise definitions and the most important properties of canonical metrics. 

We write $\opar$ for the line bundle $\oo(P)$ on $M$ equipped with the canonical Arakelov metric, and $\omar$ for the cotangent line bundle $\omega$ of $M$ equipped with the residual metric. The canonical metrics vary smoothly in families. More precisely, for the smooth proper family $\pi \colon S \to C^0$ we have a canonical smooth hermitian metric on the relative cotangent bundle $\omega_{S/C^0}$. We denote the resulting smooth hermitian line bundle on $S$ by $\omar$. Likewise, when $P \colon C^0 \to S$ is a section of $\pi \colon S \to C^0$, we have a canonical smooth hermitian metric on the line bundle $\oo(P)$ on $S$, and we denote the resulting smooth hermitian line bundle on $S$ by $\opar$. 

Our main results are then as follows. Recall that $\I$ denotes the category of all regular \emph{nc}-models of $S$ over $C$.
\begin{thm} \label{integrability} Let $P \colon C^0 \to S$ be a section of the family of curves $\pi \colon S \to C^0$.  Let the smooth hermitian line bundle $\llbar$ be either $\opar$ or $\omar$ on $S$. Let $\mmbar$ be a smooth hermitian line bundle on $C^0$ which is good on $C$ in the sense of Mumford. Assume that $S$ has semistable reduction over $C$. Then $\llbar \otimes \pi^*\mmbar$ admits all Mumford-Lear extensions in  $\I$. In particular, the b-Mumford-Lear extension $\left[ \llbar \otimes \pi^*\mmbar \right]$ exists in $\bPic_\qq(\I)$.
\end{thm}
\begin{thm} \label{chernweil} Let  $P, Q \colon C^0 \to S$ be sections of the family of curves $\pi \colon S \to C^0$. Let $\bar{\ll}_1$, $\bar{\ll}_2$ be any of the smooth hermitian line bundles $\opar$, $\oqar$ or $\omar$ on $S$. Let $\bar{\mm}_1$, $\bar{\mm}_2$ be 
smooth hermitian line bundles on $C^0$ which are good on $C$ in the sense of Mumford. Assume that $S$ has semistable reduction over $C$. Then the pair
\[ \left( \left[ \bar{\ll}_1 \otimes \pi^*\bar{\mm}_1 \right], \left[ \bar{\ll}_2 \otimes \pi^*\bar{\mm}_2 \right] \right) \]
of b-Mumford-Lear extensions in $\bPic_\qq(\I)$ is integrable. Moreover, the equality
\begin{equation} \label{inters_is_integral} \left[ \bar{\ll}_1 \otimes \pi^*\bar{\mm}_1 \right] \cdot 
 \left[ \bar{\ll}_2 \otimes \pi^*\bar{\mm}_2 \right] = 
\int_S c_1(\bar{\ll}_1 \otimes \pi^*\bar{\mm}_1 ) \wedge c_1(\bar{\ll}_2 \otimes \pi^*\bar{\mm}_2) 
\end{equation}
holds in $\qq$.
\end{thm}
We indicate how to obtain Theorem \ref{bkkthm} as a special case.  Take $C^0$ to be the modular curve $Y(N)$, where $N \geq 3$ is an integer, and let $\pi \colon E \to C^0$ be the universal elliptic curve over $C^0$. Let $C=X(N)$ denote the smooth completion of $C^0$. Let $\Lbar$ denote the line bundle $L=\oo(8O)$ on $E$ equipped with the metric that gives $\vartheta_{1,1}^8$ its Petersson norm. On the other hand let $\llbar=\oo(8O)_\Ar$ denote the line bundle $L$ on $E$ equipped with the canonical (Arakelov) metric. Let $\bar{\lambda}_{C^0}$ denote the Hodge bundle $\lambda_{C^0}$ on $C^0$, equipped with the Petersson metric, and put $\mmbar=\bar{\lambda}_{C^0}^{\otimes 4}$.  It can be shown that we then have an isometry
\[ \Lbar \isom \llbar \otimes \pi^*\mmbar   \]
of smooth hermitian line bundles on $E$, cf.\ Proposition \ref{isom} below. 

Now we note that $\bar{\lambda}_{C^0}$ is the standard automorphic line bundle on the Shimura variety $C^0$, equipped with its invariant metric. By the results of \cite{hi} the smooth hermitian line bundle $\bar{\lambda}_{C^0}$ is good on $C$ in the sense of Mumford. Also we note that $E$ has semistable reduction over $C$. Then, applying Theorems \ref{integrability} and \ref{chernweil} we conclude that $\Lbar$ has all Mumford-Lear extensions in $\I$, that the b-Mumford-Lear extension $\left[ \Lbar \right]$ in $\bPic_\qq(\I)$ is integrable, and that the equality 
\[  \left[ \Lbar \right] \cdot \left[ \Lbar \right] =  \int_E c_1(\Lbar) \wedge c_1(\Lbar)  \]
is satisfied in $\qq$.  Thus we recover Theorem \ref{bkkthm}. 

Actually, Theorems \ref{integrability} and \ref{chernweil} immediately allow for generalizations of the results in \cite{bkk} to the case of modular elliptic surfaces $E(\Gamma)$ where $\Gamma$ is any congruence subgroup of $\SL_2(\zz)$ that has no elliptic fixed points and no cusps of the second kind. We will come back to this in Section \ref{sec:modular} below. 

The proof in \cite{bkk} of equality (\ref{cwtheoryell})  consists in computing the left and right hand side of (\ref{cwtheoryell}) both explicitly. To some extent this is also true for our proof of  the equality (\ref{inters_is_integral}). In fact, what we will prove is that, in the notation of Theorem \ref{chernweil}, the Deligne pairing $\pair{\llbar_1 \otimes \pi^*\mmbar_1,\llbar_2 \otimes \pi^*\mmbar_2}$, which is naturally a smooth hermitian line bundle on $C^0$, admits a Mumford-Lear extension $\left[  \pair{\llbar_1\otimes \pi^*\mmbar_1,\llbar_2 \otimes \pi^*\mmbar_2}, C\right]$ over $C$, which then lives in $\Pic_\qq(C)$, and that both left and right hand side of equality (\ref{inters_is_integral}) are equal to the degree of the $\qq$-line bundle $\left[  \pair{\llbar_1 \otimes \pi^*\mmbar_1,\llbar_2 \otimes \pi^*\mmbar_2}, C\right]$. For the notion of Deligne pairing of two (smooth hermitian) line bundles on a family of curves we refer to~\cite[Section~XIII.5]{acg} and~\cite{de}.

We sketch some of the most important ideas in the proof, referring to the main text for details. For each $\xx \in \I$ we let $\Gamma(\xx)$ denote the metric space determined by the weighted dual graph of the boundary divisor $B(\xx)$ of $S$ in $\xx$. Let $V(\xx)\subset \Gamma(\xx)$ be the set of irreducible components of $B(\xx)$, i.e., the designated vertex set of $\Gamma(\xx)$. 

Let $\HH$ denote the direct limit $\varinjlim \Gamma(\xx)$, and let $\HH_\qq = \varinjlim V(\xx) \subset \HH$ denote the rational support of $\HH$. The space $\HH$ inherits from the $\Gamma(\xx)$ a canonical structure of metric space.  A first main ingredient of our arguments is the observation that there exists a natural isomorphism of abelian groups
\begin{equation} \label{natural}  \bDiv_\qq(\I) \isom \Div_\qq(S) \oplus \mathrm{Map}(\HH_\qq,\qq) \, . 
\end{equation}
Thus, a b-divisor on $\I$ can be seen as a pair consisting of a $\qq$-divisor $D$ on $S$ and a $\qq$-valued function $g$ on $\HH_\qq$. Without putting extra conditions on the functions $g$, it is clear that intersections of b-divisors have little chance to exist. The situation is already better if one asks for continuous extendability of $g$ over $\HH$. 
For example, the b-divisors whose projections onto the various $\Div_\qq(\xx)$ are compatible with pullback (below called the \emph{$\qq$-Cartier divisors on $\I$)} give rise to piecewise affine functions on~$\HH$. 

The space $\HH$ has a natural deformation retract $\rho \colon \HH \to \Gamma \subset \HH$, called the essential skeleton of $\HH$. The essential skeleton $\Gamma$ is a compact metrized graph. A special role will be played in our proof by functions $g$ on $\HH_\qq$ that are the pullback, along $\rho$, of the restriction to $\Gamma_\qq=\HH_\qq \cap \Gamma$ of a piecewise smooth function on $\Gamma$. 
We show that the b-Mumford-Lear extensions  $\left[ \bar{\ll}_1 \otimes \pi^*\bar{\mm}_1 \right] $ and $\left[ \bar{\ll}_2 \otimes \pi^*\bar{\mm}_2 \right] $ in $\bPic_\qq(\I)$ are classes of b-divisors of this type. The proof of this result is based on a direct connection that we will establish with the theory of compactified divisors on curves over non-archimedean valued fields as developed by S.~Zhang in \cite{zh}. The intersection pairing of our b-Mumford-Lear extensions can be explicitly calculated using the formalism in \cite{zh}. In particular one sees that the intersection pairing is well-defined and takes finite values.

We then use a result of the second author from \cite{djfaltings} to equate the Deligne pairing of the b-Mumford-Lear extensions $\left[ \bar{\ll}_1 \otimes \pi^*\bar{\mm}_1 \right] $ and $\left[ \bar{\ll}_2 \otimes \pi^*\bar{\mm}_2 \right] $ with the Mumford-Lear extension $\left[  \pair{\llbar_1 \otimes \pi^*\mmbar_1,\llbar_2 \otimes \pi^*\mmbar_2}, C\right]$. To finally equate the integral from (\ref{inters_is_integral}) with the degree of $\left[  \pair{\llbar_1 \otimes \pi^*\mmbar_1,\llbar_2 \otimes \pi^*\mmbar_2}, C\right]$ we use a lemma, which seems to be well-known, that says that a smooth hermitian line bundle $\bar{N}$ on $C^0$ that both has semipositive first Chern form on $C^0$, and has logarithmic growth on $C$, satisfies Chern-Weil theory on $C$. We then verify case by case that for the Deligne pairings $\pair{\llbar_1,\llbar_2}$ under our consideration enough semipositivity can be arranged in order to apply a mild generalization of this lemma. 

Our results show, generalizing the observations of \cite{bkk}, that in general the hermitian line bundles $\llbar_1, \llbar_2$ occurring in Theorems~\ref{integrability} and~\ref{chernweil} are \emph{not} of logarithmic growth (and in particular, not good) on any regular \emph{nc}-model of $S$. In particular, this holds for the family Arakelov metric $\omar$. This result is in sharp contrast with the situation for the family \emph{hyperbolic} metric $\omega_\mathrm{hyp}$ on $\omega$ if $g \geq 2$. Assuming semistable reduction, this metric is known to be good on any regular \emph{nc}-model by a well known result due to S.~Wolpert \cite[Theorem~5.8]{wo}. 

We mention that b-divisors have been studied and applied in a number of recent works other than \cite{bkk}. S.~Boucksom, C.~Favre and M.~Jonsson interpret in \cite{bfj-val} the functional part of b-divisors as functions on suitable valuation spaces, with the objective of developing a valuative study of complex plurisubharmonic functions and positive currents. This entails developing a good analogue of pluripotential theory in the context of Berkovich spaces \cite{bfj-sing}. In their work \cite{bfj-vol} b-divisors and products of b-divisor classes up to numerical equivalence are used to study the differentiability of the volume function for divisors on projective varieties. The work \cite{bfj-vol} reveals in particular a tight connection between positivity properties of b-divisors  and concavity properties of its functional part. 

A.~M.~Botero \cite{bo} has made an extensive study of b-divisors on smooth complete toric varieties with respect to toric blow-ups. Among other things, she shows that nef toric b-divisors correspond to  $\qq$-concave functions  and she relates the intersection product of a tuple of nef toric b-divisors to the volume of an appropriate bounded convex set. In particular, the intersection product of nef toric b-divisors is well-defined and takes finite values. We refer to the PhD thesis \cite{bo-thesis} and the forthcoming \cite{bo-toroidal} for generalizations of the results of \cite{bo} to the setting of toroidal embeddings and toroidal blow-ups. \\

The structure of this paper is as follows. In Section \ref{sec:bdiv} we first discuss preliminaries on b-divisors and their intersection product. In Section \ref{sec:ml} we show how a b-divisor can be naturally associated to a section of a smooth hermitian line bundle that admits all Mumford-Lear extensions. These preliminary  sections follow  \cite{bkk} to a large extent. From Section \ref{sec:model_skel} on we focus on the particular case of surfaces fibered over a curve. Using the notion of skeleton from \cite{bn} \cite{mn} we give a presentation of the group of b-divisors from a metric and functional perspective, leading to the isomorphism (\ref{natural}). We also briefly touch upon the connection with the theory of Berkovich analytic curves. 

In Section \ref{bpairing} we introduce the notion of metrized $\qq$-Cartier divisor, endow the group of metrized $\qq$-Cartier divisors with a natural pairing, and show that this pairing can be used to understand the intersection pairing of b-divisors. Inspired by \cite{zh} we then introduce in Section \ref{sec:comp} the notion of a compactified b-divisor, and show that the intersection pairing of compactified b-divisors exists, and can be computed using a Riemann integral over the essential skeleton. Using the notion of curvature form of compactified b-divisors we study in Section \ref{sec:admis} the subgroup of so-called admissible b-divisors. 

From Section \ref{sec:metric} on we return to the analytical side, starting with a discussion of the Arakelov metric in families of curves. Theorem \ref{integrability} is then proved in Section~\ref{sec:proof_integr}.  In Section \ref{sec:reduction} we make a reduction step that allows us to replace, in order to prove Theorem \ref{chernweil}, the category $\I$ by the smaller category $\I_0$ of so-called essential blow-ups of the minimal regular model. In Section \ref{sec:comp_ess} we then show that the b-Mumford-Lear extensions whose existence is claimed in Theorem \ref{integrability} are admissible, when computed on the category of essential blow-ups. 

The explicit intersection theory of Sections \ref{sec:comp} and \ref{sec:admis} is used in Section \ref{sec:proofcw} to give a proof of Theorem \ref{chernweil}. In Section \ref{sec:infinite} we discuss an interpretation of the identity in Theorem \ref{chernweil} as a Chern-Weil type result for hermitian line bundles with logarithmic growth on a certain canonical non-quasicompact scheme associated to the category $\I_0$. Finally in Section \ref{sec:modular} we make our results explicit for the case of modular elliptic surfaces on general congruence subgroups. \\

\noindent When $V$ is a smooth separated $\cc$-scheme (not necessarily of finite type), by a smooth hermitian line bundle on $V$ we mean a smooth hermitian line bundle on the analytification $V(\cc)$ of $V$. A complex algebraic variety is a reduced scheme which is separated and of finite type over $\cc$. 

\subsection*{Acknowledgements} We would like to thank Ana Mar\' \i a Botero, Jos\'e Burgos Gil, David Holmes and Anna-Maria von Pippich for valuable comments and discussions.

\section{b-divisors and their intersection product} \label{sec:bdiv}

We introduce the notions of b-divisors, principal b-divisors, b-line bundles, and their intersection product in a setting that will be convenient throughout the rest of the paper.  

Let $S$ be a smooth connected complex algebraic variety. A regular \emph{nc}-model of~$S$ is a smooth completion $S \hookrightarrow \xx$ of $S$ such that the boundary $\xx \setminus S$ is a normal crossings divisor on $\xx$. A map of regular \emph{nc}-models is to be a birational morphism 
$\varphi \colon \xx' \to \xx$ that restricts to the identity on $S$. We let $\I$ denote any cofiltered category of regular \emph{nc}-models of $S$ (for example, the category of \emph{all} regular \emph{nc}-models). Note that all hom-sets in $\I$ are either empty or consist of one element.

For each $\xx \in \I$ let $\Div(\xx)$ denote the group of (Cartier) divisors on $\xx$, and write $\Div_\qq(\xx)=\Div(\xx) \otimes_\zz \qq$ for the group of $\qq$-Cartier divisors on $\xx$. The group of \emph{$\qq$-Cartier divisors on $\I$} is defined to be the direct limit
\[ \cDiv_\qq(\I) = \varinjlim_{\xx \in \I} \Div_\qq(\xx)\, , \]
where all transition maps $\Div_\qq(\xx) \to \Div_\qq(\xx')$ for maps $\varphi \colon \xx' \to \xx$ in $\I$ are given by the usual pullback $\varphi^*$ of $\qq$-Cartier divisors. The notion of $\qq$-Cartier divisor on $\I$ is modelled on that of a $\qq$-Cartier divisor on the Zariski-Riemann space of $S$ as in the references \cite{bfj-val}, \cite{bfj-vol} and \cite{bkk}. 

An element of $\cDiv_\qq(\I)$ can be thought of as a $\qq$-Cartier divisor on some model $\xx\in \I$. Two such divisors on models $\xx, \xx' \in \I$  are considered to be equal in $\cDiv_\qq(\I)$ if they agree after pullback to a model $\xx'' \in \I$ that dominates both $\xx$ and $\xx'$. The group $\cDiv_\qq(\I)$ is equipped with a natural intersection product, as follows. Let $n=\dim(S)$ and suppose $\mathbb{D}_1,\ldots, \mathbb{D}_n \in \cDiv_\qq(\I)$ are $\qq$-Cartier divisors on $\I$. By our assumption that $\I$ is cofiltered, we can assume that  $\mathbb{D}_1,\ldots, \mathbb{D}_n$ are all realized on a common model $\xx \in \I$ of $S$. We then define the intersection product $\mathbb{D}_1 \cdots \mathbb{D}_n$ to be their intersection product on $\xx$. One verifies easily that this definition does not depend on the choice of common model $\xx \in \I$. 

For $\xx \in \I$ let $\Princ(\xx) \subset \Div(\xx)$ denote the set of principal divisors on $\xx$, and write $\Princ_\qq(\xx) = \Princ(\xx) \otimes_\zz \qq$. We define 
\[ \Princ_\qq(\I) = \varinjlim_{\xx \in \I} \Princ_\qq(\xx)  \subset \varinjlim_{\xx \in \I} \Div_\qq(\xx) = \cDiv_\qq(\I) \, . \]
Note that all transition maps $\varphi^* \colon \Princ_\qq(\xx) \to \Princ_\qq(\xx')$ are isomorphisms. We denote the quotient of $\cDiv_\qq(\I)$ by the subgroup $\Princ_\qq(\I) $ by $\cPic_\qq(\I)$, and the quotient map $\cDiv_\qq(\I) \to \cPic_\qq(\I)$ by $\oo$. We call $\cPic_\qq(\I)$ the group of \emph{line bundles on $\I$}. It is easily verified that the intersection product $\cdot$ on $\cDiv_\qq(\I)$ that we introduced above descends to an intersection product on $\cPic_\qq(\I)$.

The group of \emph{b-divisors on $\I$} is defined to be the projective limit
\[ \bDiv_\qq(\I) = \varprojlim_{\xx \in \I} \Div_\qq(\xx) \]
in the category of abelian groups. Here, the transition maps are given by pushforward $\varphi_* \colon \Div_\qq(\xx') \to \Div_\qq(\xx)$ of $\qq$-Weil divisors for all maps $\varphi \colon \xx' \to \xx$ in the category $\I$.  This notion of b-divisor is modelled on the notion of a Weil $\qq$-divisor on the Zariski-Riemann space of $S$ as in the references \cite{bfj-val}, \cite{bfj-vol} and \cite{bkk}. A b-divisor $\mathbb{D}$ can be seen as a tuple $\mathbb{D}=(D_\xx)_{\xx \in \I}$ of $\qq$-Weil divisors $D_\xx \in \Div_\qq(\xx)$ parametrized by the models $\xx \in \I$ with the single compatibility condition that the identity $\varphi_*D_{\xx'}=D_\xx$ should hold whenever $\varphi \colon \xx' \to \xx$ is a morphism in $\I$. 

The \emph{incarnation} or \emph{trace}  of a b-divisor $\mathbb{D}$ on any given model $\xx \in \I$ is the result of projecting $\mathbb{D}\in\bDiv_\qq(\I)=\varprojlim \Div_\qq(\xx)$ onto the factor $\Div_\qq(\xx)$. From the fact that $\varphi_* \varphi^*$ is always the identity morphism we deduce the existence of a canonical injection $\cDiv_\qq(\I) \hookrightarrow \bDiv_\qq(\I)$. In particular we have $\Princ_\qq(\I)$ naturally as a subgroup of $\bDiv_\qq(\I)$. 

We would like to extend (at least partially) the previously defined intersection product of $\qq$-Cartier divisors into an intersection product of b-divisors. Let $\mathbb{D}_1, \ldots,\mathbb{D}_n \in \bDiv_\qq(\I)$ be b-divisors on the category $\I$. We then put 
\[ \mathbb{D}_1 \cdots \mathbb{D}_n  = \lim_{\xx \in \I} D_{1,\xx} \cdots D_{n,\xx} \]
in the sense of nets, where $\mathbb{D}_i=(D_{i,\xx})_{\xx \in \I}$ for $i=1, \ldots, n$. It is easily verified that this intersection product extends the previous one of $\qq$-Cartier divisors on $\I$. An intersection product where one of the factors is a principal divisor, is clearly zero. However note that in general the limit in the definition above can not be guaranteed to exist, let alone to be finite. When the limit $\mathbb{D}_1 \cdots \mathbb{D}_n$ exists as a real number, we say that the tuple $(\mathbb{D}_1, \ldots, \mathbb{D}_n)$ of b-divisors on $\I$ is \emph{integrable}. 
 
We define the group of \emph{b-line bundles on $\I$} to be the quotient group
\[ \bPic_\qq(\I) = \bDiv_\qq(\I) / \Princ_\qq(\I) \, , \]
with quotient map denoted again by $\oo$. This quotient map clearly extends our earlier defined map $\oo \colon \cDiv_\qq(\I) \to \cPic_\qq(\I)$ between ``usual'' $\qq$-Cartier divisors and line bundles on the category $\I$. It is straightforward to verify that we have a natural identification
\[ \bPic_\qq(\I) = \varprojlim_{\xx \in \I} \Pic_\qq(\xx) \, , \]
where for each $\xx \in \I$ we write $\Pic_\qq(\xx)=\Pic(\xx) \otimes_\zz \qq$ and where on the right hand side the transition maps are given by pushforward of Weil divisor classes. Our partial intersection product of b-divisors on $\I$ descends to a partial intersection product of b-line bundles on $\I$.

\section{Mumford-Lear extensions} \label{sec:ml}

Let $S$ denote a smooth connected complex algebraic variety and let $\I$ denote any cofiltered category of regular \emph{nc}-models of $S$. Let $\bar{L}=(L,\|\cdot\|)$ be a smooth hermitian line bundle on $S$. Following \cite{bkk} we explain how a b-divisor on $\I$ naturally arises from the datum of a non-zero rational section of $L$, assuming that for each model $\xx \in \I$ the metric $\|\cdot\|$ has ``logarithmic growth away from codimension two''.  We start by explaining the latter notion. The discussion in this section is to a large extent based on  \cite[Section~3]{bkk}.

Assume that $\dim(S)=n$, and let $S \hookrightarrow \xx$ denote an open immersion, where $\xx$ is regular, and where the complement $B=\xx \setminus S$ has the structure of a normal crossings divisor. We are not assuming that $\xx$ is complete, at this point.
\begin{definition} \label{loggrowth} We say that $\bar{L}=(L,\|\cdot\|)$ \emph{has logarithmic growth on $\xx$} if the following holds: there exists a line bundle $M$ on $\xx$, together with an isomorphism $\alpha \colon M|_S \isom L$, such that for all points $p \in B $ there exists a coordinate neighborhood $(U;z_1,\ldots,z_n)$ of $p$ in $\xx$, an integer $1 \leq k \leq n$, a generating section $s \in M(U)$ and a non-negative integer $N$ such that $U \cap B$ is given by the equation $z_1 \cdots z_k =0$, and for suitable $C_1, C_2 >0$ the estimate 
\begin{equation} \label{estimate} C_1 \cdot \prod_{i=1}^k (-\log|z_i|)^{-N} \leq \| \alpha|_{U\setminus B}(s) \| \leq C_2 \cdot \prod_{i=1}^k (-\log|z_i|)^{N}  
\end{equation}
holds on $U \setminus D$.
\end{definition}
The pair $(M,\alpha)$ in Definition \ref{loggrowth} is uniquely determined up to unique isomorphism, once it exists; see for instance (the proof of) \cite[Proposition~1.3]{hi}. We denote the (isomorphism class of the) line bundle $M$ on $\xx$ by $\left[ \bar{L} , \xx \right]$, and call $\left[ \bar{L} , \xx \right]$ the \emph{Mumford extension} of $\bar{L}$. It is not difficult to show the following functoriality property. Let $S' \hookrightarrow \xx'$ be another open immersion of smooth complex varieties with $B'=\xx' \setminus S'$ a normal crossings divisor, and let $f \colon \xx' \to \xx$ be a morphism such that $f^{-1}B \subseteq B'$. Assume that $\bar{L}$ is a smooth hermitian line bundle on $S$ with logarithmic growth on $\xx$. Then $f^*\bar{L}$ is a smooth hermitian line bundle on $S'$ with logarithmic growth on $\xx'$, and we have the identity $\left[ f^*\bar{L},\xx' \right] = f^*\left[ \bar{L},\xx \right]$ in $\Pic(\xx')$. 
\begin{definition} \label{MLextension} We say that $\bar{L}$ \emph{has a Mumford-Lear extension on $\xx$} if there exists a closed subset $\Sigma \subset \xx$, which is either empty or of codimension at least two in $\xx$, and a positive integer $e$ such that the smooth hermitian line bundle $\bar{L}^{\otimes e}$ has logarithmic growth on $\xx \setminus \Sigma$. Note that by Hartogs's theorem the Mumford extension $\left[ \bar{L}^{\otimes e},\xx \setminus \Sigma \right]$ has a unique extension $M$ as a line bundle over $\xx$. We define $\left[ \bar{L},\xx \right]$ to be the element $M^{\otimes 1/e} $ of $ \Pic_\qq(\xx)$. This element is independent of the choice of $e$ and $\Sigma$, and will be called the Mumford-Lear extension of $\bar{L}$ over $\xx$.  
\end{definition}
Let $\I$ be a  cofiltered category of regular \emph{nc}-models of $S$. Recall that all maps in $\I$ are given by proper birational morphisms that restrict to the identity on $S$. Let $\bar{L}=(L,\|\cdot\|)$ be a smooth hermitian line bundle on $S$, and assume that $\bar{L}$ has a Mumford-Lear extension on \emph{all} models $\xx \in \I$. We already noted in the Introduction that the formation of the Mumford-Lear extension $\left[ \bar{L},\xx \right]$  is in general not compatible with pullback along maps $\xx' \to \xx$ in $\I$. However, the Mumford-Lear extensions do have a good functoriality behavior with respect to pushforward. More precisely, we have the following result (cf. \cite[Proposition~3.15]{bkk}).
\begin{prop} \label{create-b-line} Assume that the smooth hermitian line bundle $\bar{L}$ has a Mumford-Lear extension on all models $\xx \in \I$. The tuple of Mumford-Lear extensions $\left( \left[ \bar{L},\xx \right] \right)_{\xx \in \I}$ is then a b-line bundle on $\I$, that is, an element of $\bPic_\qq(\I)$.
\end{prop}
The proof uses two lemmas.
\begin{lem}  \label{compatibleI} Let $\varphi \colon \xx' \to \xx$ be a map in $\I$, and let $U \subset \xx$ be an open subset such that $\varphi$ is invertible on $U$ and such that the set $\xx \setminus U$ is empty or has codimension two in $\xx$. Let $U' \supset \varphi^{-1}U$ be an open subset of $\xx'$ containing $\varphi^{-1}U$ and let $i \colon U \to U'$ be the map induced from $\varphi^{-1}$ on $U$. Let $D' \in \Div_\qq(\xx')$ be a $\qq$-divisor on $\xx'$. Then the equalities $(\varphi_*D')|_{U} = i^*(D'|_{U'})$ and $\varphi_*D' = \overline{i^*(D'|_{U'})}$ hold. In particular, the diagram 
\[ \xymatrix{ \Pic_\qq(\xx') \ar[d]^{\varphi_*} \ar[r] & \Pic_\qq(U') \ar[d]^{i^*} \\
\Pic_\qq(\xx) \ar[r]^\sim    & \Pic_\qq(U) } \]
 is commutative, where the horizontal arrows are the restriction maps.
\end{lem}
\begin{proof} The first assertion follows from the chain of equalities
\[ \begin{split} i^*(D'|_{U'}) & = (\varphi^{-1})^*(D'|_{\varphi^{-1}U}) \\
 & = \varphi_*( D'|_{\varphi^{-1}U} ) \\
 & = (\varphi_*D')|_U \, .
\end{split} \]
The second assertion follows by taking Zariski closure.
\end{proof}
\begin{lem} \label{compatibleII} Let $s$ be a non-zero rational section of $L$ on $S$. Let $\varphi \colon \xx' \to \xx$ be a map in $\I$, and let $\divisor_\xx(s)$ (resp. $\divisor_{\xx'}(s)$) be the divisor of $s$ on $\xx$ (resp. $\xx'$), when viewed as a rational section of the Mumford-Lear extension $\left[ \bar{L},\xx \right]$ (resp. $\left[ \bar{L},\xx' \right]$). Then the equality $\varphi_* (\divisor_{\xx'}(s)) = \divisor_\xx(s)$ holds in $\Div_\qq(\xx)$. In particular, the tuple $\bdivisor(s) = (\divisor_\xx(s))_{\xx \in \I}$ is an element of $\bDiv_\qq(\I)$.
\end{lem}
\begin{proof} Let $U \subset \xx$ be an open subset such that $\xx \setminus U$ is empty or has codimension two, $\varphi$ is invertible on $U$, and $\bar{L}$ has logarithmic growth on $U$. We note that such an open subset exists by Zariski's Main Theorem. Let $U' = \varphi^{-1}U$ and let $i \colon U \to U'$ be the induced map. By Lemma \ref{compatibleI} we have that $\varphi_* (\divisor_{\xx'}(s)) = \overline{ i^*(\divisor_{\xx'}(s)|_{U'}  )}$. Now $\left[ \bar{L},\xx' \right]|_{U'} = \left[ \bar{L},U'\right]$ and as $\varphi|_{U'}$ is an isomorphism we have $\left[ \bar{L},U' \right] = (\varphi|_{U'})^* \left[ \bar{L},U \right]$. We derive that $i^*\left[ \bar{L},\xx' \right]|_{U'}=\left[ \bar{L},U \right]$ and more precisely the equality $i^*(\divisor_{\xx'}(s)|_{U'}) = \divisor_\xx(s)|_U$ holds. We find that $\varphi_* (\divisor_{\xx'}(s)) = \overline{ \divisor_\xx(s)|_U} = \divisor_\xx(s)$ as was to be shown.
\end{proof}
We call $\bdivisor(s) \in \Div_\qq(\I)$ the b-divisor associated to the rational section $s$ of $L$. Its image $\oo(\bdivisor(s))=(\oo(\divisor_\xx(s)))_{\xx \in \I}$ in $\bPic_\qq(\I)$ is clearly independent of the choice of $s$, and is called the b-Mumford-Lear extension of the hermitian line bundle $\bar{L}$. We denote the b-Mumford-Lear extension of $\bar{L}$ on $\I$ by $[\bar{L},\I]$, or just by $[\bar{L}]$ if the category $\I$ is clear from the context. For each $\xx \in \I$ we clearly have $ \left[ \bar{L},\xx \right] =\oo(\divisor_\xx(s))$ and from this Proposition \ref{create-b-line} immediately follows.

\section{Models of surfaces, and their skeleta} \label{sec:model_skel}

We assume from now on that our smooth connected complex algebraic variety $S$ is a surface, which is moreover equipped with a smooth proper morphism $\pi \colon S \to C^0$ onto a connected smooth complex curve $C^0$. Denote by $C^0 \hookrightarrow C$ the (unique) smooth completion of the curve $C^0$. A regular \emph{nc}-model of $S$ \emph{over} $C$ is a regular \emph{nc}-model $\xx$ of $S$ such that the morphism $\pi \colon S \to C^0$ extends (uniquely) into a morphism $\xx \to C$. We will work with a cofiltered category $\I$ all of whose objects are regular \emph{nc}-models of $S$ over $C$ (for example $\I$ could be the category of all regular \emph{nc}-models of $S$ over $C$).  Then, using the multiplicities of irreducible components in the special fibers of our models $\xx$ over the set of cusps $C\setminus C^0$, we can, following for example \cite{bn} and \cite{mn}, attach a natural metric structure $\HH$ to our category $\I$. This metric structure will help us to better understand the group of b-divisors and the partial intersection pairing on it in this particular setting. 

For each regular \emph{nc}-model $\xx$ of $S$ over $C$ we denote by $V(\xx)$ the set of irreducible components of the reduced boundary divisor $B(\xx)=\xx \setminus S$ of $S$ in $\xx$, and by $\Sigma(\xx)$ the set of singular points of $B(\xx)$. We recall that $V(\xx)$ resp.\ $\Sigma(\xx)$ are naturally the vertex set resp.\ the edge set of a natural finite graph $G(\xx)$, called the \emph{dual graph} attached to $B(\xx)$. Here for each $e \in \Sigma(\xx)$ the end-points of $e$ are the irreducible components of $B(\xx)$ on which $e$ lies. We equip the dual graph $G(\xx)$ with a structure of weighted graph as follows. Let $e$ be an edge of $G(\xx)$, with end-points $x, y \in V(\xx)$ (which could be equal). Then $e$ is assigned weight $1/mn$, where $m, n \in \zz_{>0}$ are the multiplicities of the irreducible components $x, y$ in the special fiber of the morphism $\xx \to C$ that $x, y$ belong to. From the weighted graph $G(\xx)$ we naturally obtain a metrized graph $\Gamma(\xx)$, containing $V(\xx)$ as a distinguished subset, called the \emph{skeleton} of the model $\xx$ over $C$. We note that in contrast to the usual literature, our (metrized or weighted) graphs are in general not connected. The sets of connected components of the graph $G(\xx)$ and of the metrized graph $\Gamma(\xx)$ are in natural one-to-one correspondence with the set of cusps $C \setminus C^0$.

We note that a morphism $\xx' \to \xx$ of regular \emph{nc}-models of $S$ over $C$ is a finite composition of point blow-ups. In particular, for each birational morphism $\xx' \to \xx$ of regular \emph{nc}-models of $S$ over $C$, strict transform of irreducible boundary components gives rise to a natural map of sets $V(\xx) \to V(\xx')$, and a natural map of topological spaces $\Gamma(\xx) \to \Gamma(\xx')$. As it turns out, the latter map $\Gamma(\xx) \to \Gamma(\xx')$ is an isometric embedding. Indeed, without loss of generality we may assume that $\varphi \colon \xx' \to \xx$ is the blow-up in an intersection point $e \in \Sigma(\xx)$ of two irreducible components $x, y$ in $V(\xx)$ (which may be equal). Assuming that $x, y$ have multiplicities $m, n \in \zz_{>0}$ in their fiber, the exceptional divisor $z$ of $\varphi$ has multiplicity $m+n$ in its fiber. The identity
\[  \frac{1}{mn} = \frac{1}{m(n+m)} + \frac{1}{n(m+n)}  \]
then illustrates our claim. For more details we refer \cite[Section~2]{bn}.

Let $\varphi \colon \xx' \to \xx$ be a morphism of regular \emph{nc}-models of $S$ over $C$. We then have a natural linear pullback map $\varphi^*  \colon \qq^{V(\xx)} \to \qq^{V(\xx')}$. To describe $\varphi^*$ it suffices to consider the case that $\varphi$ is a simple point blow-up. If $\varphi$ is the blow-up in a singular point of a special fiber, then $\varphi^*$ is just given by linear interpolation with respect to the natural edge weights on  $G(\xx')$ determined above. If $\varphi$ is the blow-up in a regular point of the special fiber, and $g \in \qq^{V(\xx)}$ is a map, then $\varphi^*(g)$ is just given by extending $g$ over the exceptional divisor of $\varphi$ by the value of $g$ on the irreducible component of the boundary of $\xx$ that contains the blown-up point. 

For each regular \emph{nc}-model $\xx$ of $S$ over $C$ we have an isomorphism of groups 
\[ \gamma_\xx \colon \Div_\qq(\xx) \isom \Div_\qq(S) \oplus \qq^{V(\xx)}  \]
given as follows. The first component is given by restriction of $\qq$-Cartier divisors from $\xx$ to $S$. For $D \in \Div_\qq(\xx)$ we can write $D = \overline{D|_S} + \sum_{y \in V(\xx)} a(y) y$ for uniquely determined $a(y) \in \qq$. As the second component of $\gamma_\xx(D)$ we then take the map in $\qq^{V(\xx)}$ given by $y \mapsto a(y)/v_\xx(y)$, where $v_\xx(y) \in \zz_{>0}$ is the multiplicity of $y$ in the special fiber of $\xx$. Our choice of normalization on the second component is justified by the next lemma.
\begin{lem} \label{compatible} Let $\varphi \colon \xx' \to \xx$ be a morphism of regular \emph{nc}-models of $S$ over $C$. Then the isomorphisms $\gamma_\xx$ and $\gamma_{\xx'}$ are compatible with the pullback map $\varphi^* \colon \Div_\qq(\xx) \to \Div_\qq(\xx')$ and the linear interpolation map $\varphi^* \colon \qq^{V(\xx)} \to \qq^{V(\xx')}$.
\end{lem}
\begin{proof} One quickly reduces to the case that $\varphi \colon \xx' \to \xx$ is a simple point blow-up, where the point $e$ that is blown up is an intersection point of two irreducible components $x, y \in V(\xx)$. Let $D \in \Div_\qq(\xx)$ and let $z \in V(\xx')$ denote the exceptional divisor of $\varphi$. If in $D$ component $x$ has multiplicity $a$ and $y$ has multiplicity $b$ then the multiplicity of $z$ in $\pi^*D$ is $a+b$. Assume $x$ has multiplicity $m$ in its fiber, and $y$ has multiplicity $n$ in its fiber. Then $z$ has multiplicity $m+n$ in its fiber. It follows that the second component  of $\gamma_\xx(D)$ has value $a/m$ at $x$ and $b/n$ at $y$, and that the second component of $\gamma_{\xx'}(\varphi^*D)$ has the value $(a+b)/(m+n)$ at $z$. Now recall from our constructions above that the edge between $z$ and $x$ has weight $1/m(n+m)$, and the edge between $z$ and $y$ has weight $1/n(m+n)$. We see that $(a+b)/(m+n)$ is precisely the value at $z$ that linearly interpolates between the value $a/m$ at $x$ and the value $b/n$ at $y$. 
\end{proof}

Let $\I$ be a cofiltered category of regular \emph{nc}-models of $S$ over $C$. We then put
\[ \HH= \varinjlim_{\xx \in \I} \Gamma(\xx) \, , \]
where the direct limit is taken in the category of topological spaces, along the natural embeddings $\Gamma(\xx) \to \Gamma(\xx')$. There is then a natural metric on $\HH$ for which for each $\xx \in \I$ the canonical map $\Gamma(\xx) \to \HH$ is an isometric embedding, by \cite[Theorem 2.3.3]{bn}.  We call the subset
\[ \HH_\qq = \varinjlim_{\xx \in \I} V(\xx) \subset \varinjlim_{\xx \in \I} \Gamma(\xx) = \HH  \]
the rational support of $\HH$. We put
\[ \mathrm{PA}_\qq(\HH) = \varinjlim_{\xx \in \I} \qq^{V(\xx)}  \]
in the category of $\qq$-vector spaces, where the transition maps are given by linear interpolation. Recall that $\cDiv_\qq(\I) = \varinjlim \Div_\qq(\xx)$, with transition maps given by pullback. Lemma \ref{compatible} then yields a natural isomorphism of $\qq$-vector spaces
\[  \gamma \colon \cDiv_\qq(\I) \isom  \Div_\qq(S) \oplus \mathrm{PA}_\qq(\HH) \]
by assembling all the various $\gamma_\xx$ together.

For the b-divisors on $\I$ we can say the following. Let $\varphi \colon \xx' \to \xx$ again be a morphism of regular \emph{nc}-models of $S$ over $C$. Then one verifies that the isomorphisms $\gamma_\xx$ and $\gamma_{\xx'}$ are compatible with the pushforward map $\varphi_* \colon \Div_\qq(\xx') \to \Div_\qq(\xx)$ and the natural restriction (or projection) map $\varphi_* \colon \qq^{V(\xx')} \to \qq^{V(\xx)}$.  As we can canonically write 
\[ \varprojlim_{\xx \in \I} \qq^{V(\xx)} =
\mathrm{Map}( \varinjlim_{\xx \in \I} V(\xx),\qq) = 
\mathrm{Map}(\HH_\qq,\qq)  \]
we obtain a natural isomorphism of $\qq$-vector spaces
\[  \delta \colon \bDiv_\qq(\I) \isom \Div_\qq(S) \oplus \varprojlim_{\xx \in \I} \qq^{V(\xx)} = \Div_\qq(S) \oplus \Map(\HH_\qq,\qq) \, ,   \]
characterizing b-divisors in a functional manner. Via the isomorphisms $\gamma, \delta$ the inclusion $\cDiv_\qq(\I) \hookrightarrow \bDiv_\qq(\I)$ corresponds to the inclusion of $\mathrm{PA}_\qq(\HH) $ into $\mathrm{Map}(\HH_\qq,\qq) $ as the subspace of  consisting of continuous piecewise affine functions on $\HH$ that are constant near the boundary of $\HH$, have finite corner locus supported on $\HH_\qq$, and are $\qq$-valued on $\HH_\qq$. 

Assume that the fibers of $\pi \colon S \to C^0$ have positive genus, and let $\xx_0$ denote the  minimal regular \emph{nc}-model of $S$ over $C$. As we will see below, it is useful to consider the cofiltered category $\I_0$ of all regular \emph{nc}-models $\xx$ of $S$ over $C$ for which the unique map $\xx \to \xx_0$ is a composition of point blow-ups $\varphi \colon \xx'' \to \xx'$ where each $\varphi$ is the blow-up of a singular point of the boundary divisor of $\xx'$. We call $\I_0$ the category of \emph{toroidal} or \emph{essential blow-ups}. Following \cite{bn} \cite{mn} we call the metric space 
\[ \Gamma = \varinjlim_{\xx \in \I_0} \Gamma(\xx) \]
the \emph{essential skeleton} of $S$. The essential skeleton $\Gamma$ is a compact  metrized graph, isometric to the skeleton $\Gamma(\xx_0)$. The connected components of $\Gamma$ are parametrized by the cusps of $C^0$ in $C$. Write 
\[ \Gamma_\qq =  \varinjlim_{\xx \in \I_0}  V(\xx) \, , \quad \mathrm{PA}_\qq(\Gamma) = \varinjlim_{\xx \in \I_0} \qq^{V(\xx)}  \, . \]
Then we obtain natural isomorphisms of $\qq$-vector spaces
\[  \gamma_0 \colon \cDiv_\qq(\I_0) \isom  \Div_\qq(S) \oplus \mathrm{PA}_\qq(\Gamma) \, , \, 
 \delta_0 \colon \bDiv_\qq(\I_0) \isom \Div_\qq(S) \oplus \Map(\Gamma_\qq,\qq)  \]
similar to the above $\gamma, \delta$.

To finish this section we put the above constructions into perspective by mentioning their close connection with the theory of Berkovich analytic curves. The Berkovich analytic viewpoint is the framework of the references \cite{bn},  \cite{bfj-sing}, \cite{bfj-val}, \cite{bfj-vol} and \cite{mn} mentioned above.  To fix ideas, assume that  the base curve $C^0$ has precisely one cusp~$\infty$ in its completion~$C$. Let $F$ denote the generic fiber of $\pi \colon S \to C^0$, viewed as a curve over the function field $\cc(C)$, and let $F_{\infty}$ be the curve obtained from $F$ by extending scalars to $\cc(C)_\infty$, the completion of $\cc(C)$ at the cusp $\infty$. Let $\mathfrak{S}$ denote the Berkovich analytification of the curve $F_\infty$. Then for each regular \emph{nc}-model $\xx$ of $S$ over $C$ there exists a canonical continuous retraction $\rho_\xx \colon \mathfrak{S} \to \Gamma(\xx)$, and if $\I$ denotes the category of all regular \emph{nc}-models of $S$ over $C$ these retractions piece together by \cite[Theorem~10]{ks} into a canonical isomorphism 
\[ \mathfrak{S} \isom \varprojlim_{\xx \in \I} \Gamma(\xx)   \]
of topological spaces. Via the canonical inclusion  $\HH=\varinjlim \Gamma(\xx) \hookrightarrow \varprojlim \Gamma(\xx) $ the space $\HH$ is then identified with a subspace of $\mathfrak{S}$. By \cite[Section~2.3]{bn} this subspace  is the subset of $\mathfrak{S}$ obtained by removing the points of type~I and~IV from $\mathfrak{S}$.

\section{Intersection pairing of b-divisors and skeleta} \label{bpairing}

Let $\I$ be a cofiltered category of regular \emph{nc}-models of the surface $S$ over the curve $C$. Our aim in this section is to write down the intersection pairing on the group of b-divisors $\bDiv_\qq(\I)$  in functional terms, using the metric space $\HH$ and its rational support $\HH_\qq$ from the previous section. The approach in this section is inspired by the theory developed in \cite[Section~2]{zh}. We start by giving a useful presentation of $\bDiv_\qq(\I)$ as classes of metrized $\qq$-Cartier divisors on $\I$.

Recall that we have natural isomorphisms of $\qq$-vector spaces
\[ \gamma \colon \cDiv_\qq(\I) \isom \Div_\qq(S) \oplus \mathrm{PA}_\qq (\HH)  \]
and
\[  \delta \colon \bDiv_\qq(\I) \isom  \Div_\qq(S)\oplus \mathrm{Map}(\HH_\qq,\qq) \, . \]
The map $\gamma$ can be viewed as the restriction of the map $\delta$ to $\cDiv_\qq(\I)$. Eliminating the role of the  group $\Div_\qq(S)$ we obtain an exact sequence 
\begin{equation} \label{exact} 0 \to \mathrm{PA}_\qq(\HH)  \to \cDiv_\qq(\I) \oplus \Map(\HH_\qq,\qq) \to \bDiv_\qq(\I) \to 0 
\end{equation}
where the first map is given by $g \mapsto (-\gamma^{-1}(0,g),g)$, and where the second map is given by $(\mathbb{D},g) \mapsto \mathbb{D} + \delta^{-1}(0,g)$. Drawing an analogy with the notions of metrized or adelic $\qq$-divisors, we call the vector space $\cDiv_\qq(\I) \oplus \Map(\HH_\qq,\qq)$ in the middle the space of \emph{metrized $\qq$-Cartier divisors on $\I$}. Thus, we can think of a b-divisor on $\I$ as a suitable class of metrized $\qq$-Cartier divisors on $\I$. 

Let $\Div_\qq(\HH_\qq)=\qq^{(\HH_\qq)}$ denote the group of $\qq$-divisors on $\HH_\qq$. There exists a natural (partial) symmetric bilinear pairing $i$ on  $\Div_\qq(\HH_\qq) \oplus \Map(\HH_\qq,\qq)$ given for all $D_1, D_2 \in \Div_\qq(\HH_\qq)$ and $g_1, g_2 \in \Map(\HH_\qq,\qq)$ by the prescription
\[ i(D_1 + g_1,D_2+g_2) = g_2(D_1) + g_1(D_2) + \lim_{\xx \in \I} \sum_{x, y \in V(\xx)} g_1(x)g_2(y)v_\xx(x)v_\xx(y)(x \cdot y) \, , \]
where as before $v_\xx(x), v_\xx(y) \in \zz_{>0}$ denote the multiplicities of the irreducible components $x, y \in V(\xx)$ in their special fiber of $\xx$ over $C$. Here $(x \cdot y)$ denotes the intersection product of the divisors $x, y \in V(\xx)$ on the surface $\xx$ and, as before, the limit is taken
in the sense of nets. As is the case with the pairing of b-divisors, the pairing $i$ is in general only partially defined. We claim that  the intersection pairing of b-divisors can be conveniently reformulated in terms of the pairing $i$.

First of all, note that we have a natural restriction homomorphism
\[
R \colon \cDiv_\qq(\I) \to \Div_\qq(\HH_\qq)
\]
given  by setting
\[ R(\mathbb{D}) = \sum_{x \in V(\xx)} (x \cdot \mathbb{D}_\xx)\,v_\xx(x) \, x \]
for any $\xx \in \I$ where the $\qq$-Cartier divisor $\mathbb{D}$ is realized. It is a simple check that the sum is indeed independent of the choice of $\xx$. The map $R$ vanishes on principal divisors, and hence induces a map
\[ \cPic_\qq(\I) \to \Div_\qq(\HH_\qq)  \]
that we somewhat abusively also denote by $R$.
\begin{lem} \label{inters_b_div} Let $\mathbb{D}_1 + g_1, \mathbb{D}_2 + g_2 \in \cDiv_\qq(\I) \oplus \Map(\HH_\qq,\qq)$ be two metrized $\qq$-Cartier divisors on $\I$. The intersection product of their classes in $\bDiv_\qq(\I)$ is given by the expression
\begin{equation} \label{product_expression} \begin{split} \mathbb{D}_1\cdot\mathbb{D}_2  +   i(R(\mathbb{D}_1)+g_1, & \,R(\mathbb{D}_2)+g_2) 
 = \mathbb{D}_1 \cdot \mathbb{D}_2 + g_2(R(\mathbb{D}_1)) + g_1(R(\mathbb{D}_2)) \\ &  +\lim_{\xx \in \I} \sum_{x,y \in V(\xx)} g_1(x)g_2(y)v_\xx(x)v_\xx(y)(x \cdot y)  \, . \end{split} 
 \end{equation}
In particular, the intersection product exists  if and only if the limit 
\[ \lim_{\xx \in \I} \sum_{x,y \in V(\xx)} g_1(x)g_2(y)v_\xx(x)v_\xx(y)(x \cdot y)
\]
exists.
\end{lem}
\begin{proof} Let $\xx \in \I$, and let $\mathbb{D}+g$ be a metrized $\qq$-Cartier divisor on $\I$. Then the incarnation on $\xx$ of its class  in $\bDiv_\qq(\I)$  is $\mathbb{D}_\xx + \sum_{x \in V(\xx)} g(x)\,v_\xx(x)\, x$. Now we have
\[ \begin{split} 
(\mathbb{D}_{1,\xx} & + \sum_{x \in V(\xx)}  g_1(x)\,v_\xx(x) \, x, 
\mathbb{D}_{2,\xx} + \sum_{x \in V(\xx)} g_2(x)\, v_\xx(x) \, x )  = (\mathbb{D}_{1,\xx},\mathbb{D}_{2,\xx}) \\ & + \sum_{x \in V(\xx)} g_2(x)\, v_\xx(x)\,(\mathbb{D}_{1,\xx} \cdot x)   + \sum_{x \in V(\xx)} g_1(x)\, v_\xx(x) \, (\mathbb{D}_{2,\xx} \cdot x) \\ & \hspace{1cm} + \sum_{x,y \in V(\xx)} g_1(x)g_2(y)v_\xx(x)v_\xx(y)(x \cdot y) \\
& = \mathbb{D}_1\cdot\mathbb{D}_2  + g_2(R(\mathbb{D}_1)) + g_1(R(\mathbb{D}_2)) + \sum_{x,y \in V(\xx)} g_1(x)g_2(y)v_\xx(x)v_\xx(y)(x \cdot y)  \, . 
\end{split} \]
Taking limits on the leftmost and rightmost sides gives the lemma.
\end{proof}
The limit in Lemma \ref{inters_b_div} can in many situations be conveniently rewritten using a discrete Laplacian.  Let $G$ be a \emph{simple} weighted graph (i.e., no loops and no multiple edges) with vertex set $V(G)$ and edge set $E(G)$ with weight function $\ell \colon E(G) \to \rr_{>0}$. For $x, y \in V(G)$ we set $Q(x,y)$ to be zero if there is no edge connecting $x, y$, we set $Q(x,y)=-1/\ell(e)$ if $x \neq y$ and $e \in E(G)$ is the (unique) edge connecting $x, y$, and we set $Q(x,y)=-\sum_{z \neq x} Q(x,z)$ if $x=y$. The matrix $Q=Q(x,y)$ is called the weighted Laplacian matrix of $G$. Now let  $\xx $ be a regular \emph{nc}-model of $S$ over $C$ such that the associated weighted graph $G(\xx)$ is simple. Let $x, y \in V(\xx)$ be  neighbours and let $e$ be the unique edge connecting $x, y$. Then we recall that $\ell(e)=1/v_\xx(x)v_\xx(y)$. We see that for such $\xx$ and $x, y \in V(\xx)$ we can simply write
\[ g_1(x)g_2(y)v_\xx(x)v_\xx(y)(x \cdot y)  = - g_1(x) g_2(y) Q_\xx(x,y) \, . \]
In particular, assume that the full subcategory $\I'$ of our chosen category $\I$ consisting of models $\xx$ where $G(\xx)$ is simple is cofinal in $\I$. Then we conclude that for all $D_1, D_2 \in \Div_\qq(\HH_\qq)$ and $g_1, g_2 \in \Map(\HH_\qq,\qq)$ the identity
\begin{equation} \label{with_laplacian} i(D_1 + g_1,D_2+g_2) = g_2(D_1) + g_1(D_2) - \lim_{\xx \in \I'} \sum_{x, y \in V(\xx)} g_1(x) g_2(y) Q_\xx(x,y)  
 \end{equation}
holds in $\rr$. 

Recall that we denoted by $\oo$ the canonical projection $\bDiv_\qq(\I) \to \bPic_\qq(\I)$. By slight abuse of notation we also denote by $\oo$ the induced projection map from the group of metrized $\qq$-Cartier divisors $\cDiv_\qq(\I) \oplus \Map(\HH_\qq,\qq)$ onto $\bPic_\qq(\I)$. The image $\oo(\mathbb{D}+g)$ of a metrized $\qq$-Cartier divisor $\mathbb{D}+g$ is often  somewhat suggestively denoted by $\oo(\mathbb{D})\otimes \oo(g)$, in order to stimulate thinking of elements of $\bPic_\qq(\I)$ as ``metrized $\qq$-line bundles'' in a certain sense. Our formalism then easily gives an expression for the intersection pairing of two b-line bundles $\oo(\mathbb{D}_1)\otimes \oo(g_1)$ and $\oo(\mathbb{D}_2) \otimes \oo(g_2)$ in $\bPic_\qq(\I)$. 

Actually we can be more precise. We observe that  the combinatorial term $i(R(\mathbb{D}_1)+g_1,R(\mathbb{D}_2)+g_2)$ can be canonically written as a finite sum of local contributions $i_{\mathfrak{c}}(R(\mathbb{D}_1)+g_1,R(\mathbb{D}_2)+g_2)$ where $\mathfrak{c}$ runs through the set $\mathfrak{C}$ of cusps of $C^0$ in $C$. Also we observe that for line bundles $\oo(\mathbb{D}_1), \oo(\mathbb{D}_2)$ on $\I$ we have a well-defined Deligne pairing $\pair{\oo(\mathbb{D}_1),\oo(\mathbb{D}_2)}$ in $\Pic_\qq(C)$, obtained by realizing $\mathbb{D}_1, \mathbb{D}_2$ on a common model $\xx$ of $S$ over $C$ and by applying the Deligne pairing along the unique map $\xx \to C$ extending our given map $S \to C^0$. Using these two observations we obtain a natural (partial) Deligne pairing
\[ \begin{split} \pair{\oo(\mathbb{D}_1) & \otimes \oo(g_1),\oo(\mathbb{D}_2) \otimes \oo(g_2)} \\ 
\hspace{2cm} & = \pair{\oo(\mathbb{D}_1),\oo(\mathbb{D}_2)} \otimes \oo\left(\sum_{\mathfrak{c} \in \mathfrak{C}} i_\mathfrak{c}(R(\mathbb{D}_1)+g_1,R(\mathbb{D}_2)+g_2) \cdot \mathfrak{c} \right) \end{split} \] 
in $\Pic_\rr(C)$ associated to the b-line bundles $\oo(\mathbb{D}_1) \otimes \oo( g_1)$ and $\oo(\mathbb{D}_2) \otimes \oo(g_2)$. The degree of this Deligne pairing is given by (\ref{product_expression}).

\section{Compactified b-divisors} \label{sec:comp}

In this section we consider in detail the situation where $\I$ is chosen to be the cofiltered category $\I_0$ of essential blow-ups. For our proof of Theorem \ref{chernweil} we will reduce to computing intersection pairings on this category. We recall that the associated metric space $\HH=\Gamma$ is now a (compact but not necessarily connected) metrized graph, endowed with a distinguished dense subset $\Gamma_\qq$, the rational support of $\Gamma$. We refer to \cite[Section~2]{bf}, \cite[Section~2]{cr} and \cite[Appendix]{zh} for the basic theory of compact metrized graphs that we will use here. We warn the reader that each of the references \cite{bf} \cite{cr} \cite{zh}  takes metrized graphs to be connected, but note at the same time that most definitions and statements from these references are readily generalized to the non-connected case. 

We write $\Div_\qq (\Gamma_\qq)=\qq^{(\Gamma_\qq)}$ for the group of $\qq$-divisors on $\Gamma_\qq$. For each $p \in \Gamma_\qq$ we denote by $\Dir(p)$ the finite set of directions emanating from $p$ in $\Gamma$. We denote by $\Zh(\Gamma)$ the set of continuous functions $f \colon \Gamma \to \rr$ for which there exists a finite vertex set $V_f \subset \Gamma_\qq$ containing all points of $\Gamma$ of degree $\neq 2$ such that (a) for all $p \in V_f$, and all directions $\Vec{v} \in \Dir(p)$, the directional derivative $D_{\Vec{v}}f(p)$ of $f$ in the direction of $\Vec{v}$ exists; (b) $f$ is twice differentiable on the complement of $V_f$; (c) the second order derivative $f''$ is bounded on the complement of $V_f$. By a slight abuse of terminology, we call $\Zh(\Gamma)$ the set of \emph{piecewise smooth functions} on $\Gamma$. We denote by $\Zh_\qq(\Gamma)$ the subset of $\Zh(\Gamma)$ of those piecewise smooth functions that are $\qq$-valued on $\Gamma_\qq$. For example, we have $\mathrm{PA}_\qq(\Gamma) \subset \Zh_\qq(\Gamma)$ as the set of continuous piecewise affine functions on $\Gamma$ with corner locus supported on $\Gamma_\qq$ and $\qq$-valued on $\Gamma_\qq$. The function space $\Zh(\Gamma)$ is equipped with a hybrid Laplacian operator $\Delta \colon \Zh(\Gamma) \to \Zh(\Gamma)^*$, consisting of both a continuous and a discrete part, given by 
\[  \Delta \, f = -f''(x) \, \d x -\sum_{p \in \Gamma_\qq} \sum_{\Vec{v} \in \Dir(p)} D_{\Vec{v}}f(p) \,\delta_p \]
for all $f \in \Zh(\Gamma)$. Here $\delta_p$ denotes the Dirac measure at $p$. The sum is finite, as $\sum_{\Vec{v} \in \Dir(p)} D_{\Vec{v}}f(p)$ vanishes for all $p \notin V_f$. We call a \emph{compactified $\qq$-divisor on $\Gamma$} any element of the direct sum 
\[ \overline{\Div}(\Gamma)=\Div_\qq(\Gamma_\qq)\oplus \Zh_\qq(\Gamma) \, . \]
By density of $\Gamma_\qq$ in $\Gamma$ we have a natural inclusion $\overline{\Div}(\Gamma) \subset \Div_\qq(\Gamma_\qq)\oplus \Map(\Gamma_\qq,\qq)$. In the previous section we have equipped the latter space with a partially defined symmetric bilinear pairing $i$ that served to make the intersection pairing of b-divisors explicit. We claim that the restriction of $i$ to the space $\overline{\Div}(\Gamma)$ is everywhere defined.
\begin{prop} \label{pairing_with_laplacian}
The restriction of the pairing $i$ to $\overline{\Div}(\Gamma)$ is everywhere defined. More precisely, for $D_1, D_2 \in \Div_\qq(\Gamma_\qq)$ and $g_1, g_2 \in \Zh_\qq(\Gamma)$ the identity
\[ i(D_1 + g_1,D_2+g_2) = g_2(D_1) + g_1(D_2) - \int_\Gamma g_1 \,\Delta \, g_2     \]
holds in $\rr$. 
\end{prop}
\begin{proof} 
We remark that the full subcategory $\I_0'$ of $\I_0$ of essential blow-ups $\xx$ such that $G(\xx)$ is simple is cofinal in $\I_0$. 
For each weighted graph $G(\xx)$ with $\xx \in \I_0'$ we let as before $Q_\xx$ denote the weighted Laplacian matrix of $G(\xx)$. By (\ref{with_laplacian}) we have
\[ i(D_1 + g_1,D_2+g_2) = g_2(D_1) + g_1(D_2) - \lim_{\xx \in \I_0'} \sum_{x, y \in V(\xx)} g_1(x) g_2(y) Q_\xx(x,y)  \, , \]
where the limit is taken in the sense of nets.  Now it follows from \cite[Theorem~5]{bf} or \cite[Corollary~3]{fab} that for all $g_1, g_2 \in \Zh(\Gamma)$ the equality
\[  \lim_{\xx \in \I_0'} \sum_{x, y \in V(\xx)} g_1(x) g_2(y) \, Q_\xx (x,y) = \int_\Gamma g_1\, \Delta \, g_2    \]
holds in $\rr$. The proposition follows.
\end{proof}
We define a \emph{compactified b-divisor} on $\I_0$ to be a b-divisor on $\I_0$ represented by a metrized $\qq$-Cartier divisor $\mathbb{D}+g$ on $\I_0$ where the metric part $g$ is chosen from $\Zh_\qq(\Gamma)$. The group of compactified b-divisors on $\I_0$ is denoted by $\overline{\Div}(\I_0)$. We thus have following (\ref{exact}) an exact sequence
\begin{equation} \label{exact-bis} 0 \to \mathrm{PA}_\qq(\Gamma)  \to \cDiv_\qq(\I_0) \oplus \Zh_\qq(\Gamma) \to \overline{\Div}(\I_0) \to 0 
\end{equation}
where the first map is given by $g \mapsto (-\gamma^{-1}(0,g),g)$, and where the second map is given by $(\mathbb{D},g) \mapsto \mathbb{D} + \delta^{-1}(0,g)$. The image of $\overline{\Div}(\I_0)$ in $\bPic_\qq(\I_0)$ under $\oo$ is denoted by $\overline{\Pic}(\I_0)$ and is called the group of  \emph{compactified b-line bundles}. Note that we have natural inclusions $\cDiv_\qq(\I_0) \subset \overline{\Div}(\I_0)$ and $\cPic_\qq(\I_0) \subset \overline{\Pic}(\I_0)$.

The next result follows immediately from Lemma \ref{inters_b_div} and Proposition \ref{pairing_with_laplacian}.
\begin{cor} \label{productexists} Let $\mathbb{D}_1+g_1$ and $\mathbb{D}_2+g_2$ be metrized $\qq$-Cartier divisors on $\I_0$ with $g_1, g_2 \in \Zh_\qq(\Gamma)$. Then the intersection product of the associated compactified b-divisors in $\overline{\Div}(\I_0)$ exists, and is given by 
\begin{equation} \label{product}  \mathbb{D}_1\cdot \mathbb{D}_2  + g_1(R(\mathbb{D}_2)) + g_2(R(\mathbb{D}_1))  - \int_\Gamma g_1 \, \Delta \, g_2 \, . 
\end{equation}
More precisely, the Deligne pairing $\pair{ \oo( \mathbb{D}_1)\otimes \oo(g_1), \oo(\mathbb{D}_2)\otimes\oo(g_2)}$ exists in $\Pic_\rr(C)$, and its degree is given by (\ref{product}).
\end{cor}
We observe that the natural intersection and Deligne pairing of compactified b-divisors and b-line bundles precisely match the intersection and Deligne pairing of compactified divisors and line bundles as proposed in the context of curves over complete non-archimedean fields by Zhang in \cite[Sections~2.1 and~2.2]{zh}.

Following \cite{zh} we  introduce the notion of \emph{curvature form} of a compactified b-divisor or b-line bundle, and use this notion to rewrite (\ref{product}), and to discuss positivity properties of b-divisors and b-line bundles. First of all, the curvature form on compactified divisors on $\Gamma_\qq$ is to be the linear map $c_1 \colon \overline{\Div}(\Gamma) \to \Zh(\Gamma)^*$  given by the prescription
\[c_1(D+g) = \delta_D - \Delta \,g \, . \] 
We call a compactified divisor $D+g \in \overline{\Div}(\Gamma)$ \emph{semipositive} if its curvature form $c_1(D+g) \in \Zh(\Gamma)^*$ is a positive distribution, i.e., takes non-negative values on non-negative test functions in $\Zh(\Gamma)$. 

We define the curvature form $c_1(\mathbb{D}+g)$ of a metrized $\qq$-Cartier divisor $\mathbb{D}+g$ with $g \in \Zh_\qq(\Gamma)$ to be the curvature form of the associated compactified divisor $R(\mathbb{D})+g$ on $\Gamma$, explicitly  
\[ c_1(\mathbb{D}+g)=\delta_{R(\mathbb{D})} -\Delta \, g \, . \] 
Let $\Gamma_0$ denote a connected component of $\Gamma$, and let $\deg \mathbb{D}$ denote the relative degree of the $\qq$-Cartier divisor $\mathbb{D}$ over $C$. Then  $c_1(\mathbb{D}+g)$ has total mass equal to $\deg \mathbb{D}$ on $\Gamma_0$, as $\Delta\, g$ has zero total mass over $\Gamma_0$. The intersection product in (\ref{product}) translates into  
\begin{equation} \label{with_curv_form}   \mathbb{D}_1 \cdot \mathbb{D}_2 + g_2(R(\mathbb{D}_1))  + \int_\Gamma g_1 \, c_1(\mathbb{D}_2+g_2) \, .
\end{equation}

We claim that the curvature form of metrized $\qq$-Cartier divisors $\mathbb{D}+g$ with $g \in \Zh_\qq(\Gamma)$ descends to give a curvature form of compactified b-divisors on $\I_0$. Recall from (\ref{exact-bis}) the map $\mathrm{PA}_\qq(\Gamma) \hookrightarrow \cDiv_\qq(\I_0) \oplus \Zh_\qq(\Gamma)$ given by sending $g \in \mathrm{PA}_\qq(\Gamma)$ to $-\gamma^{-1}(0,g)+g$. The following lemma then proves our claim, by (\ref{exact-bis}).
\begin{lem} \label{fund} For each element of $ \cDiv_\qq(\I_0) \oplus \Zh_\qq(\Gamma)$ lying in the image of $\mathrm{PA}_\qq(\Gamma) $, the curvature form vanishes.
\end{lem}
\begin{proof} Each element $g \in \mathrm{PA}_\qq(\Gamma) $ can be represented as a map $g \in  \qq^{V(\xx)}$ for some model $\xx \in \I_0$. 
We have $ \gamma^{-1}(0,g) = \sum_{y \in V(\xx)} g(y)\, v_\xx(y) \, y \in \Div_\qq(\xx) $ and hence
\[ R(\gamma^{-1}(0,g)) = \sum_{x, y \in V(\xx)} g(y)\, v_\xx(x)\,v_\xx(y)\, (x \cdot y) \, x \in \Div_\qq(\Gamma) \, .  \]
On the other hand we have
\[ \Delta \, g = - \sum_{x, y \in V(\xx)} g(y)\,v_\xx(x)\, v_\xx(y)\,(x\cdot y)\, \delta_x  \in \Zh(\Gamma)^* \, . \]
We conclude that $  \Delta\, g = \delta_{R(-\gamma^{-1}(0,g))} $ in $\Zh(\Gamma)^*$, and the lemma follows.
\end{proof}
For example, for a $\qq$-Cartier divisor $\mathbb{D} \in \cDiv_\qq(\I_0)$ the curvature form is just $\delta_{R(\mathbb{D})}$. In particular, for a principal divisor $\mathbb{D}=\divisor(f)$ the curvature form vanishes. Compare with \cite[Section~2.5]{zh}. It follows that the curvature form on $\overline{\Div}(\I_0)$ descends further to give a natural map
\[ c_1 \colon \overline{\Pic}(\I_0) \to \Zh(\Gamma)^* \, . \]
We call a compactified b-divisor or b-line bundle \emph{semipositive} if its curvature form is a positive distribution. For $\mathbb{D} \in \cDiv_\qq(\I_0)$ one has that $\mathbb{D}$ is semipositive if and only if $\mathbb{D}$ is relatively nef, which shows that our notion of semipositivity for compactified b-divisors is reasonable. It would be interesting to compare our notions of curvature form and semipositivity with the notions of curvature form and semipositivity as developed in \cite[Sections~4 and~5]{bfj-sing}  in the context of pluripotential theory on Berkovich analytic spaces.

\section{Admissible b-divisors} \label{sec:admis}

We continue to work with the category $\I_0$ of essential blow-ups and its associated metrized graph $\Gamma$. An important role in the sequel is played by so-called admissible b-divisors and b-line  bundles. These are special compactified b-divisors and b-line bundles whose curvature form has a given shape. Following \cite{zh} we introduce a class of special distributions in $\Zh(\Gamma)^*$ called admissible measures. We can work component by component in $\Gamma$ and we will assume from now on, to ease the notation, that  $\Gamma$ is connected. Also we assume in this section that the minimal regular model $\xx_0$ of $S$ over $C$ is semistable.

Let $K_\can \in \Div_\qq(\Gamma_\qq)$ denote the canonical divisor of $\Gamma$, i.e.\ the divisor given by $K_\can(x) = v(x)-2$ for each $x \in \Gamma_\qq$ where $v(x)$ denotes the degree of $x$. Note that $K_\can$ is well-defined since $v(x) \neq 2$ for only finitely many points $x \in \Gamma_\qq$. Assume that a vertex set $V \subset \Gamma_\qq$ is chosen such that $V$ contains all points of $\Gamma_\qq$ of degree $\neq 2$. The complement of $V$ in $\Gamma$ can now be written as a finite union of open line segments, that we call the edge set $E$ associated to $V$. Each $e \in E$ has a length or  weight $\ell(e) \in \qq_{>0}$. The canonical measure on $\Gamma$ is by definition the distribution
\[   \mu_\can = -\frac{1}{2}\delta_{K_{\can}} + \sum_{e \in E} \frac{ \d\, x}{\ell(e) + r(e)}  \]
in $\Zh(\Gamma)^*$, where $r(e)$ denotes the effective resistance between the endpoints of $e$ in the graph obtained by removing the edge $e$ from $\Gamma$. Here $r(e)$ is set to be $\infty$ if upon removal of $e$ the graph $\Gamma$ becomes disconnected. The measure $\mu_\can$ is independent of the choice of $V$, and has total mass equal to one, by \cite[Theorem~2.11]{cr}. The canonical measure arises naturally in the study of the variational properties of the effective resistance on $\Gamma$, by \cite{cr}. 

In order to define the admissible measures we fix a rational number $h \neq 0$ and pick a $K \in \Div_\qq(\Gamma_\qq)$ of degree $2h-2$. Zhang's admissible measure with respect to $K$ is then the element
\[ \mu = \mu_K= \frac{1}{2h}\left( \delta_K + 2 \,\mu_\can \right) \] 
of $\Zh(\Gamma)^*$. Note that $\mu$ again has total mass equal to one. Let $x \in \Gamma_\qq$ be a point. Then Zhang's Arakelov-Green's function relative to $K$ is the element of $\Zh(\Gamma)$ uniquely determined by the two conditions
\begin{equation} \label{def_green} \Delta_y \, g_\mu(x,y) =\delta_x(y) - \mu(y) \, , \quad \int_\Gamma g_\mu(x,y) \, \mu(y) = 0 \, . 
\end{equation}

We refer to \cite[Section~3 and Appendix]{zh} for a proof of existence in $\Zh(\Gamma)$ of the function $g_\mu(x,y)$ for all $x \in \Gamma$. The uniqueness is clear as the kernel of $\Delta$ consists of only the constant functions on $\Gamma$. The proof in \cite{zh} actually shows that each $g_\mu(x,y)$ is a piecewise quadratic function on $\Gamma$, and it is not difficult to see that $g_\mu$ is $\qq$-valued on $\Gamma_\qq \times \Gamma_\qq$ as on each segment $e \in E$ the function $g_\mu(x,y)$ for $x \in \Gamma_\qq$ fixed can be expressed as a quadratic polynomial with rational coefficients. In particular we have  for $x \in \Gamma_\qq$ fixed that $g_\mu(x,y) \in \Zh_\qq(\Gamma)$. By \cite[Theorem~3.2]{zh} there exists $c=c(\Gamma,K) \in \qq$ such that for all $y \in \Gamma$ the identity
\begin{equation} \label{constant}
c + g_\mu(y,y) + g_\mu(K,y) = 0  
\end{equation}
holds. We refer to \cite{ci} \cite{zh} for more details.

We say that a compactified b-divisor determined by a metrized $\qq$-Cartier divisor $\mathbb{D}+g \in \cDiv_\qq(\I_0) \oplus \Zh_\qq(\Gamma)$ is \emph{admissible} if its curvature form $c_1(\mathbb{D}+g)$ is a multiple of $\mu$. Comparing volumes we see that this multiple is then necessarily equal to $\deg \mathbb{D}$, the relative degree of $\mathbb{D}$ over $C$. Likewise we say that a compactified b-line bundle $\oo(\mathbb{D})\otimes \oo(g)$ is \emph{admissible} if its curvature form is a multiple of $\mu$. We denote the group of admissible b-divisors by $\overline{\Div}_a(\I_0)$, and the group of admissible b-line bundles by $\overline{\Pic}_a(\I_0)$.

Let $\omega \in \cPic_\qq(\I_0)$ denote the line bundle on $\I_0$ determined by the relative dualizing sheaf of the minimal regular model $\xx_0$ of $S$ over $C$. We take $K=R(\omega) \in \Div_\qq(\Gamma_\qq)$, and we will fix this choice from now on. The rational number $h=\frac{1}{2}(\deg K +2)$ now equals the genus of the fibers of $\pi \colon S \to C^0$, and is in particular non-zero. Moreover it is not difficult to check that the associated admissible measure  $\mu$ is now a positive distribution in $\Zh(\Gamma)^*$. 

Following again \cite{zh} we have the following special admissible b-line bundles on $\I_0$. First of all we put
\[  \omega_a = \omega \otimes \oo( c + g_\mu(K,y)) = \omega \otimes \oo(-g_\mu(y,y)) \, , \]
where $g_\mu$ is Zhang's Arakelov-Green's function and where $c$ is the constant defined in (\ref{constant}).  Next, let $P \colon C^0 \to S$ be a section of the projection $\pi \colon S \to C^0$. Let $\oo(\bar{P}) \in \Pic_\qq(\I_0)$  denote the line bundle on $\I_0$ determined by taking the Zariski closure in any model of the section $P$, and let $x=R(\oo(\bar{P}))$ be the associated one-point divisor on $\Gamma_\qq$. Thus for any model $\xx$ the vertex $x \in V(\xx) \subset \Gamma_\qq$ is the unique irreducible component to which $\bar{P}$ specializes. Recall that we are still assuming, for ease of exposition, that the graph $\Gamma$ is connected, i.e.\ that $C^0$ has precisely one cusp in $C$. We put 
\[  \oo(P)_a = \oo(\bar{P}) \otimes \oo( g_\mu(x,y) ) \, . \]
One verifies easily that $c_1( \oo(P)_a) = \mu$ and $c_1(\omega_a) = \deg(\omega)\,\mu$, so that both $\oo(P)_a$ and $\omega_a$ are semipositive admissible b-line bundles. When $P, Q \colon C^0 \to S$ are sections of the projection $\pi \colon S \to C^0$ we obtain from (\ref{with_curv_form}) the equality
\begin{equation} \label{inters_adm}  \oo(P)_a \cdot \oo(Q)_a = \oo(\bar{P})  \cdot \oo(\bar{Q}) + g_\mu(x, y) 
\end{equation}
in $\qq$ if $\bar{P}$ specializes to $x \in \Gamma_\qq$ and $\bar{Q}$ specializes to $y \in \Gamma_\qq$.

\section{The Arakelov metric} \label{sec:metric}

The purpose of this section is to introduce the canonical Arakelov metric in the complex analytic setting, and to discuss its main properties. For more details and background we refer to \cite{ar}, \cite{fa} and \cite{djfaltings}. We start by fixing a complete connected complex algebraic curve $M$ of positive genus $g \geq 1$. Let $\omega$ denote the line bundle of holomorphic differentials on $M$. Then the space $\omega(M)$ of global sections of $\omega$ has a natural hermitian inner product given by the assignment
\[ (\alpha,\beta) \mapsto \frac{i}{2} \int_M \alpha \wedge \bar{\beta} \, . \]
Let $(\alpha_1,\ldots,\alpha_g)$ denote an orthonormal basis of $\omega(M)$. The canonical volume form of $M$ is then defined to be the $(1,1)$-form
\[ \mu = \frac{i}{2g} \sum_{j=1}^g \alpha_j \wedge \bar{\alpha}_j \]
on $M$. One easily verifies that $\mu$ has total mass equal to one, is independent of the choice of orthonormal basis, and is indeed nowhere vanishing.  

The Arakelov-Green's function on $M$ is the generalized function $g_\Ar$ on $M \times M$ determined by the conditions
\[  \deldelbar_z g_\Ar(P,z) = \pi i\left( \mu(z) - \delta_P(z) \right) \, ,\quad \int_M g_\Ar(P,z) \mu(z) = 0  \]
for all $P \in M$. Note the similarity with (\ref{def_green}). For a proof of existence of $g_\Ar$ we refer to \cite{ar}; uniqueness is clear by compactness of $M \times M$. An application of Stokes's theorem yields the symmetry relation $g_\Ar(P,Q)=g_\Ar(Q,P)$ for all $P, Q \in M$, and \cite{ar} shows that for all $P \in M$ the generalized function  $g_\Ar(P,z)$ is smooth on $M \setminus \{P\}$, and develops a logarithmic singularity at $P$. 

The Arakelov-Green's function $g_\Ar$ induces a smooth hermitian metric $\|\cdot\|$ on the line bundle $L=\oo(\Delta)$ on $M \times M$, where $\Delta$ denotes the diagonal, by setting $\log \|1\|(P,Q)=g_\Ar(P,Q)$. Here $1$ denotes the canonical global section of $L$. By restriction to horizontal or vertical slices of $M \times M$ we obtain natural smooth hermitian metrics on the line bundles $\oo(P)$ for all $P \in M$. We denote the line bundle $\oo(P)$ equipped with the canonical metric induced by $g_\Ar$ by $\opar$. We note that $c_1(\opar)=\mu$.

The Poincar\'e residue induces a canonical adjunction isomorphism of line bundles
\begin{equation} \label{adjunction}   \oo_{M \times M}(-\Delta)|_\Delta \isom \omega  
\end{equation}
on $M$. By this isomorphism one obtains a canonical smooth hermitian residual metric $\|\cdot\|_\Ar$ on $\omega$. We use the symbol $\omar$ to refer to the line bundle $\omega$ equipped with the metric $\|\cdot\|_\Ar$. It is proved in \cite[Section~4]{ar} that $c_1(\omar)=\deg(\omega)\mu$.

The Arakelov metrics vary smoothly in families. More precisely, let $Y$ denote a smooth connected algebraic variety, and let $\pi \colon S \to Y$ be a proper smooth morphism with connected fibers of dimension one and of genus $g \geq 1$. Let $\omega=\omega_{S/Y}$ denote the relative cotangent line bundle of $S$ over $Y$. Equipping $\omega$ fiberwise with the residual metric $\|\cdot \|_\Ar$ we obtain the structure of a smooth hermitian line bundle on $\omega$ over $S$, that we will denote again by $\omar$. Let $P \colon Y \to S$ be a section of the projection $\pi \colon S \to Y$. Then equipping $\oo(P)$ fiberwise with the Arakelov metric derived from $g_\Ar(P,\cdot)$ we obtain the structure of a smooth hermitian line bundle on $\oo(P)$ over $S$, that we will denote by $\opar$.

We refer to \cite[Sections 5--7]{djfaltings} for more background on these ``family'' Arakelov metrics. We will need the following facts in our proof of Theorem \ref{chernweil}. 
Let $J \to Y$ with zero section $e \colon Y \to J$ be the family of jacobians associated to the family of curves $S \to Y$. Let $\bb$ denote the Poincar\'e bundle on $J$, equipped with its canonical translation-invariant smooth hermitian metric $\|\cdot\|$, cf.\ for example \cite[Section~2.2]{bghdj_sing} or \cite[Section~3]{hain_biext} for constructions and properties. Let $P, Q \colon Y \to S$ denote two sections of $\pi \colon S \to Y$, and let  $\delta \colon Y \to J$ denote the induced Abel-Jacobi map sending a point $z \in Y$ to the class of the degree-zero divisor $P_z-Q_z$ in $J_z$. Then as a special case of (7.6) in \cite[Section~7]{djfaltings} we have a canonical isometry
\begin{equation}  \label{deligne_poincare} \delta^*\bb \isom \pair{ \oo(P-Q)_\mathrm{Ar}, \oo(P-Q)_\mathrm{Ar} }^{\otimes -1} 
\end{equation}
of smooth hermitian line bundles on $Y$. 

Let $Y \hookrightarrow X$ denote an open immersion of smooth algebraic varieties, such that the boundary $X \setminus Y$ of $Y$  in $X$ is a normal crossings divisor on $X$.
\begin{thm} \label{lear_exists} The smooth hermitian line bundle $\delta^*\bb $ is semipositive on $Y$, and has a Mumford-Lear extension over $X$. 
\end{thm}
\begin{proof} We observe that $\delta \colon Y \to J$ is an admissible normal function of the family of (intermediate) jacobians $J \to Y$ in the sense of asymptotic Hodge theory. The first statement then follows from \cite[Theorem~13.1]{hain_normal} or \cite[Theorem~8.2]{pp}. The second statement follows from \cite[Theorem~6.6]{lear}, \cite[Theorem~6.1]{hain_normal} or \cite[Theorem~5.19]{pearldiff}. 
\end{proof}

\section{Proof of Theorem \ref{integrability}} \label{sec:proof_integr}

Actually we prove a more general version of Theorem \ref{integrability}, see Theorem \ref{admissibility} below.

Let $S$ be a smooth connected complex algebraic variety and let $T \to S$ be a smooth proper curve of positive genus. Let $P, Q$ be sections of $T \to S$. Let $S\hookrightarrow \xx$ be an open immersion into a smooth complex variety $\xx$ with boundary a normal crossings divisor $B$.  Denote by $B^\sing$ the set of singular points of $B$.
\begin{thm}  \label{leararakelov}
Assume that $T \to S$ extends into a semistable curve $\mathcal{T} \to \xx$. Then for each of the smooth hermitian line bundles $\pair{ \opar, \omar }$, $\pair{\opar,\oqar}$ and $\pair{\omar,\omar}$ on $S$ it holds that some positive tensor power has a continuous extension over $\xx \setminus B^\sing$. In particular each of $\pair{ \opar, \omar }$, $\pair{\opar,\oqar}$ and $\pair{\omar,\omar}$ has a Mumford-Lear extension over $\xx$.
\end{thm}
\begin{proof} This follows from \cite[Theorem~1.3]{djfaltings}.
\end{proof}
Let $Y$ be a smooth complex variety and let $Y \hookrightarrow Z$ be an open immersion into a smooth complex variety $Z$ with boundary a normal crossings divisor $D$. Let $\pi \colon S \to Y$ be a smooth proper curve of positive genus. Let $P \colon Y \to S$ be a section of $S \to Y$. Let $\llbar$ be either the smooth hermitian line bundle $\opar$ or $\omar$ on $S$. Let $\mmbar$ be a smooth hermitian line bundle on $Y$ with logarithmic growth on $Z$. 

 \begin{thm} \label{admissibility} Let $\xx \to Z$ be a regular \emph{nc}-model of $\pi \colon S \to Y$ and denote by $B$ the reduced boundary divisor of $S$ in $\xx$. Assume that $\pi$ extends into a semistable curve $\xx_0 \to Z$.
Then some positive tensor power of the smooth hermitian line bundle $\llbar \otimes \pi^*\mmbar$ has logarithmic growth on $\xx \setminus B^\sing$. In particular $\llbar \otimes \pi^*\mmbar$ has a Mumford-Lear extension over $\xx$. 
\end{thm}
\begin{proof} Put $T = S \times_Y S$, and let $p_1, p_2 \colon T \to S$ denote the projections on the first and second coordinate, respectively. Let $\Delta \colon S \to T$ denote the diagonal map. We then have canonical isometries $ \omar \isom \pair{\oo(\Delta)_{\mathrm{Ar}},p_2^*\omar} $ and $\opar \isom \pair{\oo(\Delta)_{\mathrm{Ar}}, p_2^*\oo(P)_{\mathrm{Ar}}} $
of smooth hermitian line bundles on $S$, where both Deligne pairings are taken along $p_1$ and where we consider $\Delta$ as a section of $p_1$. We apply Theorem \ref{leararakelov} with $\mathcal{T} = \xx \times_Z \xx_0$. Note that the projection $\mathcal{T} \to \xx$ is semistable and extends the map $p_1 \colon T \to S$. We deduce  that some positive tensor power of both $\omar$ and $\opar$ has a continuous extension over $\xx \setminus B^\sing$. By functoriality $\pi^*\mmbar$ has logarithmic growth on $\xx$, and the theorem follows.
\end{proof}
Applying Proposition \ref{create-b-line} we obtain the following corollary.
\begin{cor} \label{bmlexists} Let $\I$ be any cofiltered category of regular \emph{nc}-models of $S \to Y$. Then $ \llbar \otimes \pi^*\mmbar$ admits all Mumford-Lear extensions in $\I$. In particular, the b-Mumford-Lear extension $\left[ \llbar \otimes \pi^*\mmbar, \I \right]$ of $\llbar \otimes \pi^*\mmbar$ exists in $\bPic_\qq(\I)$.
\end{cor}
Note that Corollary \ref{bmlexists} has Theorem \ref{integrability} as a special case.

\section{Reduction to the essential category} \label{sec:reduction}

Let again $C^0$ be a smooth connected complex algebraic curve, and let $C$ denote its smooth completion. Let $\pi \colon S \to C^0$ be a smooth proper curve over $C^0$ of positive genus. Let $\I_0$ be the category of essential blow-ups determined by $\pi \colon S \to C^0$. Recall that the objects of $\I_0$ are precisely the regular \emph{nc}-models of $S$ over $C$ where the unique map to the minimal regular model $\xx_0$ is a composition of point blow-ups $\varphi \colon \xx'' \to \xx'$ where each $\varphi$ is the blow-up of a singular point of the boundary divisor $B(\xx')$ of $S$ in $\xx'$. The purpose of  this section is to show that in order to prove Theorem \ref{chernweil}, we may as well restrict to the category $\I_0$ and compute the intersection product on the left hand side of (\ref{inters_is_integral}) on $\I_0$. 

As in Theorem \ref{chernweil}, let $P, Q \colon C^0 \to S$ be sections of the map $\pi \colon S \to C^0$. Let $\llbar_1$, $\llbar_2$ be any of $\opar$, $\oqar$ or $\omar$ on $S$, and let $\mmbar_1, \mmbar_2$ be smooth hermitian line bundles on $C^0$ that are good on $C$ in the sense of Mumford. We put $\bar{L}_1=\llbar_1 \otimes \pi^*\mmbar_1$ and $\bar{L}_2=\llbar_2 \otimes \pi^*\mmbar_2$. 

\begin{prop} \label{can_restrict} Let $\I$ be the category of all regular \emph{nc}-models of $S$ over $C$. The identity
\[ \left[ \bar{L}_1, \I  \right] \cdot \left[ \bar{L}_2, \I  \right] = \left[ \bar{L}_1, \I_0  \right] \cdot \left[ \bar{L}_2, \I_0  \right] \]
holds in $\rr$. More precisely, assume that the intersection product on one side exists in $\rr$. Then so does the intersection product on the other side, and both intersection products are equal. 
\end{prop}
\begin{proof} Corollary \ref{bmlexists} yields that all b-Mumford-Lear extensions for $\bar{L}_1$ and $\bar{L}_2$ exist on both $\I$ and $\I_0$. We need to show the identity
\[ \lim_{\xx \in \I}  \left[ \bar{L}_1,\xx \right] \cdot \left[ \bar{L}_2,\xx \right] = \lim_{\xx \in \I_0}  \left[ \bar{L}_1,\xx \right] \cdot \left[ \bar{L}_2,\xx \right] \]
in $\rr$, where both limits are taken in the sense of nets. More precisely, assuming the limit on one side exists in $\rr$, we need to show that the limit on the other side exists in $\rr$, and that the two limits are equal.
It suffices to verify the following: if $\varphi \colon \xx' \to \xx$ is a simple blow-up in a point not in $B(\xx)$, then 
$\left[ \bar{L}_1,\xx' \right] \cdot \left[ \bar{L}_2,\xx' \right] = \left[ \bar{L}_1,\xx \right] \cdot \left[ \bar{L}_2,\xx \right]$. By Theorem \ref{admissibility} we know that both $\left[ \bar{L}_1,\xx \right]$ and $\left[ \bar{L}_2,\xx \right]$ have logarithmic growth on $\xx \setminus B^\sing$, up to taking a positive tensor power. This implies that $\left[ \bar{L}_1,\xx' \right]=\varphi^*\left[ \bar{L}_1,\xx \right]$ and $\left[ \bar{L}_2,\xx' \right]=\varphi^*\left[ \bar{L}_2,\xx \right]$ under our assumption on $\varphi$. We deduce that 
\[ \begin{split} \left[ \bar{L}_1,\xx' \right] \cdot \left[ \bar{L}_2,\xx' \right]  & = \varphi^*\left[ \bar{L}_1,\xx \right] \cdot \varphi^*\left[ \bar{L}_2,\xx \right] \\
& = \left[ \bar{L}_1,\xx \right] \cdot \left[ \bar{L}_2,\xx \right]
\end{split} \]
by the projection formula.
\end{proof}

\section{Computing on the essential category} \label{sec:comp_ess}

We continue with the setting of the previous section, and work on the essential category $\I_0$. We assume that $S \to C^0$ has semistable reduction over $C$. Let $P \colon C^0 \to S$ be a section of the map $\pi \colon S \to C^0$. Let $\llbar$ be one of the smooth hermitian line bundles $\opar$ or $\omar$ on the complex surface $S$. Let $\mmbar$ be a smooth hermitian line bundle on $C^0$ which is good on $C$ in the sense of Mumford. We prove a result that characterizes the b-Mumford-Lear extension $\left[ \llbar \otimes \pi^*\mmbar \right] $ in $\bPic_\qq(\I_0)$ as an admissible line bundle in $\bPic_\qq(\I_0)$ (see Section~\ref{sec:admis}). 

Let $Q \colon C^0 \to S$ be another section of the map $\pi \colon S \to C^0$. Let $\bar{N}$ be any of the smooth hermitian line bundles $\pair{ \opar, \omar }$, $\pair{\opar,\oqar}$ or $\pair{\omar,\omar}$ on $C^0$, where the Deligne pairings are taken along the map $\pi$. Then by Theorem~\ref{leararakelov} each $\bar{N}$ has a Mumford-Lear extension $\left[ \bar{N},C \right]$ over $C$. Results from \cite{djfaltings} imply that one can give a precise description of these Mumford-Lear extensions in terms of the Deligne pairing on the group $ \overline{\Pic}_a(\I_0)$ of admissible b-line bundles.

Let $\Gamma$ be the essential skeleton of $S$ over $C$, endowed with its canonical structure of metrized graph. Let $\oo(\bar{P}) \in \cPic_\qq(\I_0)$ resp.\ $\omega \in \cPic_\qq(\I_0)$ be the canonical classes determined by the section $P$ of $S \to C^0$ resp.\ the relative dualizing sheaf $\omega$ of the minimal regular model $\xx_0$ over $C$, and put $x=R(\oo(\bar{P}))$ and $K=R(\omega)$  in $\Div_\qq(\Gamma_\qq)$. We then have the admissible b-line bundles 
\[  \oo(P)_a = \oo(\bar{P}) \otimes \oo( g_\mu(x,y)) \, , \quad \omega_a =\omega \otimes \oo(-g_\mu(y,y)) \]
in $\bPic_\qq(\I_0)$, where $g_\mu$ is Zhang's Arakelov-Green's function with respect to the admissible measure $\mu=\mu_K$ on $\Gamma$. Any Deligne pairing of two b-line bundles among $  \oo(P)_a$, $\oo(Q)_a$ and $ \omega_a$ exists and takes values in $\Pic_\qq(C)$, by Corollary \ref{productexists}.
\begin{thm} \label{maindj}  
Assume that $S \to C^0$ has semistable reduction over $C$. Then the equalities 
\[ \begin{split}  [\pair{ \opar, \omar },C] &= \pair{\oo(P)_\a,\omega_\a} \, ,\\
  [\pair{\opar,\oqar},C] &= \pair{\oo(P)_\a,\oo(Q)_\a} \, , \\
  [\pair{ \omar, \omar },C] &= \pair{\omega_\a,\omega_\a} \end{split} \]
hold in $\Pic_\qq(C)$.
\end{thm}
\begin{proof} This is \cite[Theorem~1.4]{djfaltings}.
\end{proof}
In fact, we can identify the b-Mumford-Lear extensions $\left[ \opar \right]$ resp.\ $\left[ \omar \right]$ on $\I_0$ themselves with $\oo(P)_a$ resp.\ $ \omega_a$ in $\bPic_\qq(\I_0)$.  We will show this now. Note that we have a map $\Pic_\qq(C) \to \cPic_\qq(\I_0)$ given by pullback, and hence a natural map $\Pic_\qq(C) \to \bPic_\qq(\I_0)$. We have a Mumford extension $[\mmbar,C]\in \Pic_\qq(C)$ of $\mmbar$ over $C$, and by a slight abuse of notation we will also denote by $[\mmbar,C]$ its image in $\bPic_\qq(\I)$. 
\begin{thm} \label{characterize} Let $\llbar$ be either the smooth hermitian line bundle $\opar$ or $\omar$ on the surface $S$ over $C^0$. Let $\mmbar$ be a smooth hermitian line bundle on $C^0$ which is good on $C$ in the sense of Mumford. Assume that $S \to C^0$ has semistable reduction over $C$. Then the b-Mumford-Lear extension $ \left[ \llbar \otimes \pi^*\mmbar \right] \in \bPic_\qq(\I_0) $ on $\I_0$  is an admissible b-line bundle on $S$. More precisely, the identities 
\[ \left[\omar \otimes \pi^*\mmbar \right] = \omega_a \otimes [\mmbar,C] \quad \textrm{and} \quad \left[ \opar  \otimes \pi^*\mmbar \right] = \oo(P)_a   \otimes [\mmbar,C] \]
hold in $\bPic_\qq(\I_0)$. 
\end{thm}
\begin{proof}  We recall that by Corollary \ref{bmlexists} the b-Mumford-Lear extension $\left[ \llbar \otimes \pi^*\mmbar \right] $ exists in $\bPic_\qq(\I_0)$. Pick a regular $\emph{nc}$-model $\xx \in \I_0$. Let  $B=\xx \setminus S$ be its reduced boundary divisor and let $G$ denote the dual graph of $B$, with vertex set $V(G)$ and edge set $E(G)$. 
We let  $m, n \in \qq^{V(G)}$ be determined by the conditions
\[ [\omar,\xx]= \omega \otimes \oo( \sum_{y \in V(G)} m(y) y) \, , \quad 
[\opar,\xx] = \oo(\bar{P}) \otimes \oo( \sum_{y \in V(G)} n(y) y )\, .  \]
Let $y \in V(G)$. In order to compute the coefficients $m(y), n(y)$ it suffices to study the behavior of the metric on $\omar$ and $\opar$ locally in the complex topology near any point $p$ of $y \setminus B^\sing$. Here, the actual choice of $p \in y \setminus B^\sing$ is irrelevant since $y \setminus B^\sing$ is connected in the complex topology. Thus, let $(Z,0)$ be a pointed smooth connected complex curve and let $f \colon Z \to \xx$ be a morphism with $f(0)=p \in y \setminus B^\sing$, intersecting $y$ transversally, and such that $Z^*=Z \setminus \{0\}$ is mapped into $S$.  
Let $T = S \times_{C^0} S$, and let $p_1, p_2 \colon T \to S$ denote the projections on the first and second coordinate, respectively. Let $\Delta \colon S \to T$ denote the diagonal map. As we saw earlier in the proof of Theorem \ref{admissibility} we then have canonical isometries
\[ \omar \isom \pair{\oo(\Delta)_{\mathrm{Ar}},p_2^*\omar} \, , \quad 
\opar \isom \pair{\oo(\Delta)_{\mathrm{Ar}}, p_2^*\oo(P)_{\mathrm{Ar}}} \]
of smooth hermitian line bundles on $S$, where both Deligne pairings are taken along $p_1$ and where we consider $\Delta$ as a section of $p_1$.
We apply Theorem \ref{maindj} above to the base change $T \times_S Z^* \to Z^*$ along $f|_{Z^*} \colon Z^* \to S$ and the smooth completion $Z$ of $Z^*$. As the image $f(Z)\subset \xx$ avoids $B^\sing$ we find that the formation of the Mumford-Lear extension and pulling back along $f$ commute. This yields the equalities
\[  m(y) = -g_\mu(y,y)\,v_\xx(y) \quad \textrm{and} \quad n(y) = g_\mu(x,y)\,v_\xx(y)   \]
in $\qq$, where $x \in V(G)$ is the unique irreducible component of $B$ that lies in the same fiber over $C$ as $y$ and intersects $\bar{P}$. We thus obtain
\[ \label{finitelevel} [\omar,\xx] = \omega \otimes \oo(-\sum_{y \in V(G)} g_\mu(y,y)\,v_\xx(y) \,y )\]
and
\[ [\opar,\xx] = \oo(\bar{P}) \otimes \oo( \sum_{y \in V(G)} g_\mu(x,y)\,v_\xx(y)\, y ) \]
in $\Pic_\qq(\xx)$. We can now vary $\xx$ through $\I_0$. As $\Gamma_\qq \subset \Gamma$ is dense we deduce that 
\[ \left[\omar \right] =\omega \otimes \oo(-g_\mu(y,y)) =  \omega_a  \]
and similarly 
\[   \left[ \opar \right]  = \oo(\bar{P}) \otimes \oo(g_\mu(x,y))=\oo(P)_a \]
in $\bPic_\qq(\I_0)$. Finally it is straightforward to check that the identity $[\mmbar,C]=\left[ \pi^*\mmbar,\I \right]$ holds in $\bPic_\qq(\I_0)$.
The theorem follows. 
\end{proof}

We continue to let $\pi \colon S \to C^0$ denote a smooth proper curve of positive genus over the curve $C^0$, and $P, Q \colon C^0 \to S$ two sections of $\pi$. Let $\bar{\ll}_1$, $\bar{\ll}_2$ be any of the smooth hermitian line bundles $\opar$, $\oqar$, $\omar$ on $S$, and let $\bar{\mm}_1$, $\bar{\mm}_2$ be smooth hermitian line bundles on $C^0$ which are good on $C$ in the sense of Mumford. 
\begin{thm} \label{pairing_is_extension} Assume that $S$ has semistable reduction over $C$. The Deligne pairing $\pair{ \left[ \bar{\ll}_1 \otimes \pi^*\bar{\mm}_1 \right], \left[ \bar{\ll}_2 \otimes \pi^*\bar{\mm}_2 \right]}$ of b-Mumford-Lear extensions exists in $\Pic_\qq(C)$. More precisely, the equality
\begin{equation} \label{Deligne_is_ML}  \pair{ \left[\bar{\ll}_1 \otimes \pi^*\bar{\mm}_1 \right],\left[ \bar{\ll}_2 \otimes \pi^*\bar{\mm}_2 \right]} = \left[
\pair{ \bar{\ll}_1 \otimes \pi^*\bar{\mm}_1 , \bar{\ll}_2 \otimes \pi^*\bar{\mm}_2 } , C\right] 
\end{equation}
holds in $\Pic_\qq(C)$.
\end{thm}
\begin{proof}  Combining Theorem \ref{maindj}  and Theorem \ref{characterize}  we find the equalities
\[  \pair{ \left[ \omar \right], \left[ \omar \right] }  = \pair{ \omega_a,\omega_a} = \left[ \pair{\omar,\omar} , C\right] \, ,  \]
and similarly
\[  \pair{ \left[ \opar \right], \left[ \omar\right] } = \pair{\oo(P)_a,\omega_a} = \left[ \pair{\opar,\omar  }, C\right]  \]
and
\[ \pair{  \left[ \opar \right], \left[ \oqar\right] }  = \pair{ \oo(P)_a,\oo(Q)_a}  = \left[ \pair{\opar, \oqar }, C \right]   \]
in $\Pic_\qq(C)$. We conclude that 
\[ \begin{split} \pair{ \left[\bar{\ll}_1 \otimes \pi^*\bar{\mm}_1 \right],\left[ \bar{\ll}_2 \otimes \pi^*\bar{\mm}_2 \right]}  & =
 \pair{ \left[\bar{\ll}_1 \right]\otimes \left[\bar{\mm}_1 , C\right],\left[ \bar{\ll}_2 \right] \otimes \left[\bar{\mm}_2 , C\right]}  \\
 & = \pair{ \left[ \llbar_1 \right ], \left[ \llbar_2 \right]} \otimes \left[\mmbar_1,C\right]^{\deg \ll_2} \otimes \left[\mmbar_2,C\right]^{\deg \ll_1} \\
 & = \left[ \pair{\llbar_1,\llbar_2} \right] \otimes \left[\mmbar_1,C\right]^{\deg \ll_2} \otimes \left[\mmbar_2,C\right]^{\deg \ll_1} \\
 & = \left[\pair{ \bar{\ll}_1 \otimes \pi^*\bar{\mm}_1 , \bar{\ll}_2 \otimes \pi^*\bar{\mm}_2 } , C\right] 
\end{split} \]
in $\Pic_\qq(C)$. This proves the theorem.
\end{proof}

\section{Proof of Theorem \ref{chernweil}} \label{sec:proofcw}

We can now prove Theorem \ref{chernweil}, which we reformulate for convenience.
\begin{thm} \label{chernweil_bis}  Assume that $S$ has semistable reduction over $C$. Let $\I$ denote the category of all regular \emph{nc}-models of $S$ over $C$. The intersection number
\[ \left[ \bar{\ll}_1 \otimes \pi^*\bar{\mm}_1 ,\I\right] \cdot 
 \left[ \bar{\ll}_2 \otimes \pi^*\bar{\mm}_2 , \I  \right] \]
of b-Mumford-Lear extensions in $\bPic_\qq(\I)$ exists and is an element of $\qq$. Moreover, the equality
\[ \left[ \bar{\ll}_1 \otimes \pi^*\bar{\mm}_1 ,\I\right] \cdot 
 \left[ \bar{\ll}_2 \otimes \pi^*\bar{\mm}_2 , \I \right] = 
\int_S c_1(\bar{\ll}_1 \otimes \pi^*\bar{\mm}_1 ) \wedge c_1(\bar{\ll}_2 \otimes \pi^*\bar{\mm}_2) 
\]
holds in $\qq$.
\end{thm}
First of all, by Lemma \ref{can_restrict} we can compute the intersection product of b-Mumford-Lear extensions by restricting to the category $\I_0$. Theorem \ref{pairing_is_extension} shows then by taking degrees on left and right hand side in (\ref{Deligne_is_ML}) that the equality
\[   \left[ \bar{\ll}_1 \otimes \pi^*\bar{\mm}_1,\I \right] \cdot  \left[ \bar{\ll}_2 \otimes \pi^*\bar{\mm}_2 , \I \right] =  \deg \left[
\pair{ \bar{\ll}_1 \otimes \pi^*\bar{\mm}_1 , \bar{\ll}_2 \otimes \pi^*\bar{\mm}_2 } , C\right]  \]
holds in $\qq$, from which the first statement of the theorem follows. The task we are left with is therefore to show that the equality
\[   \deg \left[ \pair{ \bar{\ll}_1 \otimes \pi^*\bar{\mm}_1 , \bar{\ll}_2 \otimes \pi^*\bar{\mm}_2 } , C\right]  = \int_S c_1(\bar{\ll}_1 \otimes \pi^*\bar{\mm}_1 ) \wedge c_1(\bar{\ll}_2 \otimes \pi^*\bar{\mm}_2)  \]
holds in $\qq$. This is now a purely analytic question.

First of all we have a chain of equalities
\[ \begin{split} \int_S c_1(\bar{\ll}_1 \otimes \pi^*\bar{\mm}_1 ) \wedge c_1(\bar{\ll}_2 \otimes \pi^*\bar{\mm}_2) & = \int_{C^0} \int_\pi c_1(\bar{\ll}_1 \otimes \pi^*\bar{\mm}_1 ) \wedge c_1(\bar{\ll}_2 \otimes \pi^*\bar{\mm}_2) \\
& = \int_{C^0} c_1(  \pair{\bar{\ll}_1 \otimes \pi^*\bar{\mm}_1, \bar{\ll}_2 \otimes \pi^*\bar{\mm}_2 } ) \\
 & = \int_{C^0} c_1(\pair{\bar{\ll}_1,\bar{\ll}_2}) + \deg \ll_2 \int_{C^0} c_1(\bar{\mm}_1) \\ & \hspace{1cm} +  \deg \ll_1 \int_{C^0} c_1(\bar{\mm}_2) \, . \end{split} \]
The second equality holds by \cite[Proposition~6.6]{de}.
By Mumford's results \cite{hi} as discussed in the Introduction the good hermitian line bundles $\mmbar_1, \mmbar_2$ satisfy Chern-Weil theory on $C$, in particular the equalities
\[ \int_{C^0} c_1(\bar{\mm}_1) = \deg \, [\bar{\mm}_1,C] \, , \quad 
\int_{C^0} c_1(\bar{\mm}_2) = \deg \, [\bar{\mm}_2,C]  \]
hold in $\zz$. It follows that we have
\[ \begin{split}  \int_S c_1(\bar{\ll}_1 \otimes \pi^*\bar{\mm}_1 ) \wedge c_1(\bar{\ll}_2 \otimes \pi^*\bar{\mm}_2) & = \int_{C^0} c_1(\pair{\bar{\ll}_1,\bar{\ll}_2}) + \deg \ll_2  \deg \, [\bar{\mm}_1,C]  \\ & \hspace{1cm} + \deg \ll_1   \deg \, [\bar{\mm}_2,C]  \, . \end{split} \]
On the other hand we have
\[ \begin{split} \deg \,[ \pair{ \bar{\ll}_{1} \otimes \pi^*\bar{\mm}_1, \bar{\ll}_2 \otimes \pi^*\bar{\mm}_2},C] & = \deg \,[ \pair{\bar{\ll}_{1},\bar{\ll}_{2}},C] + \deg \ll_2 \deg \, [\bar{\mm}_1 , C] \\ & \hspace{1cm} + \deg \ll_1 \deg \, [\bar{\mm}_2 , C]  \, .  \end{split} \]
This means that our task is reduced to showing the equality
\begin{equation} \label{int=deg_forL} \deg \,[ \pair{\bar{\ll}_{1},\bar{\ll}_{2}},C] = \int_{C^0} 
c_1(\pair{\bar{\ll}_1,\bar{\ll}_2}) 
\end{equation}
in $\qq$. But this equality is guaranteed by the next two lemmas.
\begin{lem} \label{acg_lemma} Let $\bar{N}=(N,\|\cdot\|)$ be a smooth hermitian line bundle on the smooth curve $C^0$. Assume that the first Chern form $c_1(\bar{N})$ is semipositive on $C^0$, and that $\bar{N}$ has a Mumford-Lear extension $[\bar{N},C] \in \Pic_\qq(C)$.  Then the equality
\begin{equation} \label{integral=degree} \int_{C^0} c_1(\bar{N}) = \deg \, [\bar{N},C] 
\end{equation}
holds in $\qq$. More generally, by additivity one has equation (\ref{integral=degree}) in $\qq$ if merely some positive tensor power of $\bar{N}$ can be written as the quotient of two smooth hermitian line bundles on $C^0$ that both have semipositive first Chern form, and that both have a Mumford-Lear extension over $C$.
\end{lem}
\begin{proof} This follows immediately from \cite[Lemma~XI.9.17]{acg}.
\end{proof}
\begin{remark} Note that, since $C^0$ is a curve, the hermitian line bundle $\bar{N}$ has a Mumford-Lear extension over $C$ if and only if $\bar{N}$ has logarithmic growth on $C$.
\end{remark}
\begin{lem} As above let $\bar{\ll}_1$, $\bar{\ll}_2$ be any of the smooth hermitian line bundles $\opar$, $\oqar$, $\omar$ on $S$. Then some positive tensor power of the Deligne pairing $\pair{\llbar_1,\llbar_2}$ on $C^0$ can be written as the quotient of two smooth hermitian line bundles on $C^0$ that both have semipositive first Chern form, and that both have a Mumford-Lear extension over $C$.
\end{lem}
\begin{proof} We recall that by Theorem~\ref{leararakelov} or Theorem~\ref{maindj} each of the Deligne pairings $\pair{ \opar, \omar }$, $\pair{\opar,\oqar}$ and $\pair{\omar,\omar}$ has a Mumford-Lear extension on $C$. 
Now it is shown in \cite[Proposition~4.7]{bhdj} and \cite[Corollary~4.8]{bhdj} that $ \pair{\omar,\omar }$ and $\pair{\opar,\omar} $
are semipositive on $C^0$. We therefore immediately obtain the statement of the lemma for these two Deligne pairings. As we have a canonical isometry $ \pair{\opar,\opar} \isom \pair{\opar,\omar}^{\otimes -1} $ by adjunction (\ref{adjunction}), the lemma also follows for the pairing $ \pair{\opar,\opar}$. 

It remains to deal with the Deligne pairing $\pair{\opar,\oqar}$. Following the end of Section \ref{sec:metric}, let $J \to C^0$ be the jacobian fibration associated to the family of curves $S \to C^0$. Let $\bb$ be the Poincar\'e bundle on $J$, equipped with its canonical translation-invariant metric, and let $\delta \colon C^0 \to J$ denote the Abel-Jacobi map determined by the sections $P, Q$. By (\ref{deligne_poincare}) we have a canonical isometry
\[  \delta^*\bb \isom \pair{ \oo(P-Q), \oo(P-Q) }^{\otimes -1}  \]
of smooth hermitian line bundles on $C^0$. By expanding the Deligne brackets on the right hand side and using again adjunction (\ref{adjunction}) we find a canonical isometry
\[ \pair{\opar,\oqar}^{\otimes 2} \isom \delta^*\bb \otimes \pair{\opar,\omar}^{\otimes -1} \otimes \pair{\oqar,\omar}^{\otimes -1} \, . \]
Now by Theorem \ref{lear_exists}, the smooth hermitian line bundle $\delta^*\bb$ is semipositive, and has a Mumford-Lear extension over $C$. As we saw earlier the same properties hold for the smooth hermitian line bundle $ \pair{\opar,\omar} \otimes \pair{\oqar,\omar}   $.
We conclude that $\pair{\opar,\oqar}^{\otimes 2}$ can be written as a quotient of two smooth hermitian line bundles, both of which are semipositive and have a Mumford-Lear extension over $C$.
\end{proof}

\section{Infinite Mumford extensions} \label{sec:infinite}

The purpose of this section is to discuss an interpretation of Theorem \ref{chernweil} as a Chern-Weil type result for line bundles with logarithmic growth on a certain ``non-quasicompact compactification'' $V$ of the surface $S$.  Let $\I_0$ be the cofiltered category of all essential blow-ups of the minimal regular model of $S$. For all $\xx \in \I_0$ we write $U(\xx)=\xx \setminus B(\xx)^\sing$ for the complement of the singular locus of the boundary divisor. Then we note that for all maps $\varphi \colon \xx' \to \xx$ in $\I_0$ we have $U(\xx') \supset \varphi^{-1}U(\xx)$. Indeed, by definition each $\varphi \colon \xx' \to \xx$ in $\I_0$ is a finite composition of blow-ups in boundary singular points, and thus maps boundary singular points into boundary singular points.  For $\varphi \colon \xx' \to \xx$ a map in $\I_0$ we denote by $i \colon U(\xx) \hookrightarrow U(\xx')$ the induced map in the opposite direction. As $B(\xx')^\sing$ is either empty or has codimension two Lemma~\ref{compatibleI} implies that the inverse system $\Pic_\qq(\xx)_{\xx \in \I_0}$ with transition maps $\varphi_*$ and the inverse system $\Pic_\qq(U(\xx))_{\xx \in \I_0}$ with transition maps $i^*$ are canonically isomorphic. In particular, we have a canonical isomorphism
\[  \psi \colon \varprojlim_{\xx \in \I_0} \Pic_\qq(U(\xx))  \isom \bPic_\qq(\I_0)   \]
of $\qq$-vector spaces, where on the left hand side the transition maps are given by pullback of $\qq$-line bundles. Now let $V = \varinjlim_{\xx \in \I_0} U(\xx)$, the direct limit taken in the category of locally ringed spaces. Then it is readily verified that $V$ is a $\cc$-scheme, containing $S$ naturally as an open dense subscheme, with smooth boundary $V \setminus S$, and with the set of irreducible components of the boundary $V \setminus S$ canonically in bijection with the countably infinite set $\Gamma_\qq=\varinjlim V(\xx)$. It is natural to define a $\qq$-line bundle on $V$ to be the datum of an open cover $\{ V_i \}_{i \in I}$ of $V$ with quasicompact schemes $V_i$ plus $\qq$-line bundles on all $V_i$, together with isomorphisms of $\qq$-line bundles on all overlaps, satisfying the cocycle condition on triple intersections. Write $\Pic_\qq(V)$ for the group of $\qq$-line bundles on $V$. 
We have a natural map
 \[  \rho \colon \Pic_\qq(V) \to \varprojlim_{\xx \in \I_0} \Pic_\qq(U(\xx)) \]
of $\qq$-vector spaces, given by restriction of $\qq$-line bundles along the inclusions $U(\xx) \subset V$. We contend that the map $\rho$ is actually an isomorphism. First of all, a $\qq$-line bundle on $V$ is determined by its restrictions to all $U(\xx)$ which shows that $\rho$ is injective. Next, to give an element of $\varprojlim \Pic_\qq(U(\xx))$ is to give $\qq$-line bundles $\ll_\xx$ on all $U(\xx)$ compatible with all inclusions $U(\xx) \hookrightarrow U(\xx')$ coming from maps $\xx' \to \xx$ in $\I_0$. For  maps $\xx'' \to \xx$ and $\xx'' \to \xx'$ in $\I_0$ the overlap $U(\xx) \cap U(\xx')$ can be viewed inside $U(\xx'')$ and we see that for $\ll_\xx$ and $\ll_{\xx'}$ on $U(\xx)$ and $U(\xx')$ coming from $\ll_{\xx''}$ on $U(\xx'')$ the restrictions to  $U(\xx) \cap U(\xx')$ coincide. Now via the composed isomorphism $\psi \circ \rho$ the partial intersection pairing $\cdot$ on $\bPic_\qq(\I_0)$ can be viewed as a partial intersection pairing $\cdot$ on $\Pic_\qq(V)$.

We continue with the notation introduced in Theorem \ref{chernweil}. Theorem~\ref{admissibility} implies that for each $\xx \in \I_0$  the smooth hermitian line bundles $\bar{\ll}_1 \otimes \pi^*\bar{\mm}_1$ and $\bar{\ll}_2 \otimes \pi^*\bar{\mm}_2$ both have logarithmic growth on $U(\xx)$.   Then $\bar{\ll}_1 \otimes \pi^*\bar{\mm}_1$ and $\bar{\ll}_2 \otimes \pi^*\bar{\mm}_2$ both have logarithmic growth on $V= \varinjlim U(\xx)$, and their Mumford canonical extensions over all $U(\xx)$ glue together to give $\qq$-line bundles $\left[\bar{\ll}_1 \otimes \pi^*\bar{\mm}_1,V\right]$ and $\left[\bar{\ll}_2 \otimes \pi^*\bar{\mm}_2 , V\right]$ on the scheme $V$.  We can then rewrite equality  (\ref{inters_is_integral}) in Theorem~\ref{chernweil} more conceptually as
\begin{equation} \label{proposal} \left[\bar{\ll}_1 \otimes \pi^*\bar{\mm}_1,V\right] \cdot \left[\bar{\ll}_2 \otimes \pi^*\bar{\mm}_2 , V\right] = 
\int_V c_1(\bar{\ll}_1 \otimes \pi^*\bar{\mm}_1 ) \wedge c_1(\bar{\ll}_2 \otimes \pi^*\bar{\mm}_2) \, . 
\end{equation}
Thus the scheme $V$ should be viewed as a ``compactification'' of $S$ on which Chern-Weil theory holds naturally for the line bundles $\bar{\ll}_1 \otimes \pi^*\bar{\mm}_1$ and $\bar{\ll}_2 \otimes \pi^*\bar{\mm}_2$ with logarithmic growth.

\section{Modular elliptic surfaces} \label{sec:modular}

Let $\Gamma \subset \SL_2(\zz)$  be a subgroup of finite index, and let $Y=Y(\Gamma)$ be the associated modular curve. We assume that $\Gamma$ acts without fixed points on the upper half plane~$\HH$ and that $Y$ has no cusps of the second kind. In this final section we make the result of Theorem \ref{chernweil} explicit in the case that $S$ is the total space $E=E(\Gamma)$ of the universal elliptic curve $\pi \colon E \to Y$ over $Y$. Our assumptions on $\Gamma$ imply that  $E$ has semistable reduction over the complete modular curve $X=X(\Gamma)$. 

Let $O \colon Y \to E$ denote the zero-section of $\pi \colon E \to Y$. Following \cite{bkk} let $\Lbar $ denote the line bundle $L=\oo_{E}(8 O )$ on $E$ equipped with the hermitian metric that gives the global section $\vartheta_{1,1}^8$ of $L$ its Petersson metric $\|\cdot\|_{\mathrm{Pet}}$. Our assumptions on $\Gamma$ imply that the metric on $\Lbar$ is smooth. Let  $\llbar = \oo_{E}(8O)_{\mathrm{Ar}} $ denote the line bundle $L$ equipped with the canonical Arakelov metric, let $\bar{\lambda}_{Y}$ denote the Hodge bundle on $Y$ equipped with the Petersson metric, and let  $\mmbar=\bar{\lambda}_{Y}^{\otimes 4}$. By Mumford's results in \cite{hi} we have that $\bar{\lambda}_{Y}$ is a good hermitian line bundle on $X$.  
\begin{prop} \label{isom} There exists an isometry
\[ \Lbar \isom \llbar \otimes \pi^*\mmbar  \]
of smooth hermitian line bundles on $E$.
\end{prop}
\begin{proof} Let $\eta(\tau)$ for $\tau \in \HH$ be Dedekind's eta-function, and let $g_\tau$ denote the Arakelov-Green's function associated to the origin on the elliptic curve $\cc/(\zz+\tau \zz)$. By \cite[Section~7]{fa} we have that for $(z,\tau) \in \cc \times \mathbb{H}$ the equality
\[ \log \|\vartheta_{1,1}\|_{\mathrm{Pet}}(z,\tau)  = g_\tau(0,z)  + \log \|\eta\|_{\mathrm{Pet}}(\tau)  \]
is satisfied. The proposition follows by observing that $\eta(\tau)^{24}=\Delta(\tau)$ is a global section of the line bundle $\lambda_{Y}^{\otimes 12}$. 
\end{proof}
We find the following result as a special case of Theorems \ref{integrability} and \ref{chernweil}. Let $\I$ denote the category of all regular \emph{nc}-models of $E$ over $X$.
\begin{thm} \label{modular} The smooth hermitian line bundle $\Lbar$ admits all Mumford-Lear extensions in the category $\I$. The resulting b-Mumford-Lear extension $ \left[ \Lbar \right] \in  \bPic_\qq(\I) $ is integrable, and the equality
\begin{equation} \label{cwth_ell_general} \left[ \Lbar \right] \cdot \left[ \Lbar \right] = \int_{E(\Gamma)} c_1(\Lbar) \wedge c_1(\Lbar)  
\end{equation}
holds in $\qq$.
\end{thm}
It is instructive to compute the intersection product  $\left[ \Lbar \right] \cdot \left[ \Lbar \right]$ explicitly and thus, via  (\ref{cwth_ell_general}), to generalize formula (\ref{integral}) due to Burgos Gil, Kramer and K\"uhn to the case of general $\Gamma$. Let $\ee$ denote the minimal regular model of $E$ over $X$.
Let $\Delta$ be the discriminant divisor (i.e., divisor of cusps) on $X$ and put $ d_\Gamma = \deg \Delta $. We then have $d_\Gamma=12 \deg \lambda_X$ where $\lambda_X$ is the Hodge bundle of $\ee$ over $X$.
The following result was proved by the first named author by a different method in his master thesis \cite{je}. 
\begin{thm} The equality
\[  \left[ \Lbar \right] \cdot \left[ \Lbar \right] =  \frac{16}{3}d_\Gamma \]
holds in $\qq$.
\end{thm}
\begin{proof}
First of all we have
\[    \left[ \Lbar \right] \cdot \left[ \Lbar \right] =  \left[ \llbar \right] \cdot \left[ \llbar \right] + 16 \deg  \left[ \mmbar,X \right]  \]
since $\llbar$ has relative degree~$8$ over $Y$.
Now we have $\left[ \mmbar,X \right]=4\,\lambda_X$ by  \cite[p.~225]{fc} and this gives
\[ \left[ \Lbar \right] \cdot \left[ \Lbar \right]  =  \left[ \llbar \right] \cdot \left[ \llbar \right] + 64 \deg \lambda_X = \left[ \llbar \right] \cdot \left[ \llbar \right]  + \frac{16}{3}d_\Gamma \, . \]
We finish by showing that $\left[ \llbar \right] \cdot \left[ \llbar \right] =0 $. First of all, by Theorem~\ref{characterize} we have
\[  \left[ \llbar \right]  = 8 \, \oo(O)_a \, . \]
Write $o=R(\bar{O})$. By definition of $\oo(O)_a$ we have
\[ \oo(O)_a = \oo(\bar{O}) \otimes \oo(\sum_{\mathfrak{c} \in \mathfrak{C}} g_{\mathfrak{c},\mu}(o,\cdot))\, , \]
where $\mathfrak{C}$ denotes the reduced divisor of cusps of $X$, and $g_{\mathfrak{c},\mu}$ is Zhang's Arakelov-Green's function on the connected metrized graph $\Gamma_{\mathfrak{c}}$ obtained from taking the direct limit, over all models in $\I_0$, of the dual graphs of the fiber over $\mathfrak{c} \in \mathfrak{C}$.  By equation (\ref{inters_adm})  we have
\[  \oo(O)_a \cdot \oo(O)_a = \oo(\bar{O}) \cdot \oo(\bar{O}) + \sum_{\mathfrak{c} \in \mathfrak{C}} g_{c,\mu}(o,o) \, . \]
Let $\mathfrak{c} \in \mathfrak{C}$ be a cusp. Then the metrized graph $\Gamma_\mathfrak{c}$ is a circle with origin~$o$ and of length $\ell_\mathfrak{c} = \ord_\mathfrak{c}(\Delta)$, the multiplicity of the discriminant divisor $\Delta$ at $c$.
We have by \cite[Section a.8]{zh} for $x, y \in \Gamma_\mathfrak{c}$ explicitly that
\[ g_{\mathfrak{c},\mu}(x,y) = \frac{1}{2\ell_\mathfrak{c}} (t(x)-t(y))^2 - \frac{1}{2}|t(x)-t(y)| + \frac{1}{12}\ell_\mathfrak{c} \, , \]
where $t$ is a euclidean coordinate on $\Gamma$ with $0 \leq t < \ell_\mathfrak{c}$. We find $g_{\mathfrak{c},\mu}(o,o)=\frac{1}{12}\ell_\mathfrak{c}$ and hence
\[ \sum_{\mathfrak{c} \in \mathfrak{C}} g_{\mathfrak{c},\mu}(o,o) = \frac{1}{12} \sum_{\mathfrak{c} \in \mathfrak{C}}  \ell_\mathfrak{c} = \frac{1}{12} \deg \Delta = \deg \lambda_X \, . \]
On the other hand, let $\omega$ denote the relative dualizing sheaf of the minimal regular model $\ee$ of $E$ over $X$. Then one has
\[  \oo(\bar{O}) \cdot \oo(\bar{O}) = \oo_{\ee}(O)\cdot\oo_{\ee}(O)  = - \oo_{\ee}(O) \cdot \omega= -\deg \lambda_X \, ,  \] 
where the middle equality holds by adjunction and where the last equality holds since $\omega = \pi^* \lambda_X$.
We conclude that $ \oo(O)_a \cdot \oo(O)_a=0$ and this completes the proof.
\end{proof}
\begin{cor}
The equality
\[ \int_{E(\Gamma)} c_1(\Lbar) \wedge c_1(\Lbar) = \frac{16}{3}d_\Gamma  \]
holds in $\qq$.
\end{cor}
For $\Gamma = \Gamma(N)$ we have $d_\Gamma = Np_N$ where $p_N$ is the number of cusps of $X(N)$ and thus we re-obtain (\ref{integral}).

\vspace{0.5cm}

\noindent Addresses of the authors:\\ \\

\noindent Michiel Jespers \\
Email: \verb+jespersc@xs4all.nl+ \\ \\

\noindent Robin de Jong \\
Mathematical Institute  \\
Leiden University \\
The Netherlands  \\
Email: \verb+rdejong@math.leidenuniv.nl+

\end{document}